\documentclass{gOPT2e}

\usepackage{amssymb,amsmath,amsthm, amsfonts}

\newcommand{\at}[1]{}
\newcommand{\email}[1]{{\tt #1}}

\newcommand{\Gr}{{\rm gph\,}}

\newcommand{\co}{{\rm conv\,}}
\newcommand{\cl}{{\rm cl\,}}

\newcommand{\yb}{\bar y}

\newcommand{\wb}{\bar w}
\newcommand{\yba}{\yb^\ast{}}
\newcommand{\lb}{\bar\lambda}

\newcommand{\I}{{\cal I}}
\newcommand{\J}{{\cal J}}
\newcommand{\Null}{{\cal N}}

\newcommand{\Kb}{{\bar K}}%{K(\yb,\yba)}
\newcommand{\Tlin}{T_\Gamma^{\rm lin}}
\newcommand{\Ib}{\bar\I}
\newcommand{\Jb}{\bar\J}
\newcommand{\Lb}{\bar\Lambda}
\newcommand{\Mb}{\bar{\cal M}}
\newcommand{\Zb}{\bar Z}
\newcommand{\PI}{P_{I^+,\I}}
\newcommand{\PIt}{P_{\tilde I^+,\tilde \I}}
\newcommand{\KbI}{\bar K_{I^+,\I}}
\newcommand{\KbIt}{\bar K_{\tilde I^+,\tilde \I}}

\newcommand{\R}{\mathbb{R}}
\newcommand{\norm}[1]{\|#1\|}
\newcommand{\Norm}[1]{\left\|#1\right\|}
\newcommand{\dist}[1]{{\rm d}(#1)}
\newcommand{\B}{{\cal B}}
\newcommand{\E}{{\cal E}}

\newcommand{\mv}{\,\vert\, }

\newcommand{\oo}{o}
\newcommand{\AT}[2]{{\textstyle{#1\atop#2}}}
\newcommand{\skalp}[1]{\langle #1\rangle}

\newcommand{\range}{{\rm Range\,}}

\newcommand{\Limsup}{\mathop{{\rm Lim}\,{\rm sup}}}

\newtheorem{definition}{Definition}
\newtheorem{theorem}{Theorem}
\newtheorem{lemma}{Lemma}
\newtheorem{example}{Example}

\newtheorem{proposition}{Proposition}
\newtheorem{remark}{Remark}

\begin{document}

\title{ On computation of limiting coderivatives of the normal-cone mapping to inequality systems and their applications\thanks{This is an Accepted Manuscript of an article published by Taylor \& Francis in Optimization on 20 July 2015, available online: http://www.tandfonline.com/10.1080/02331934.2015.1066372}}
\author{Helmut Gfrerer$^{\rm a}$$^\ast$\thanks{$^\ast$ Corresponding author. Email \email{helmut.gfrerer@jku.at}} and
         Ji\v{r}\'{i} V. Outrata$^{\rm b}$\\
         $^{\rm a}$Institute of Computational Mathematics, Johannes Kepler University Linz,
          Linz, Austria; $^{\rm b}$Institute of Information Theory and Automation, Academy of Sciences of the Czech Republic, Prague, Czech Republic, and Centre for Informatics and Applied Optimization, School of Science, Information Technology and Engineering, Federation University of Australia, Ballarat, Australia.}

%\date{}

\maketitle
\begin{abstract}The paper concerns the computation of the limiting coderivative  of the normal-cone mapping related to $C^{2}$ inequality constraints under  weak qualification conditions. The obtained results are applied to verify the Aubin property of solution maps to a class of parameterized generalized equations.
\end{abstract}
\begin{classcode} 49J53, 90C31, 90C46.
\end{classcode}

\section{Introduction}

In sensitivity and stability analysis of parameterized optimization and equilibrium problems via the tools of modern variation analysis one often needs to compute the limiting (Mordukhovich) normal cone to the graph of the mapping $ \hat{N}_{\Gamma} (\cdot)$, where $ \hat{N}_{\Gamma} $ stands for the regular (Fr\'{e}chet) normal cone to a closed (not necessarily convex) constraint set $ \Gamma $. This research started in the nineties with the paper \cite{DR96}, where the authors obtained an exact formula for the above mentioned limiting normal cone in the case when $ \Gamma $ is a convex polyhedron. The special case of $ \Gamma $ being the nonnegative orthant paved then the way to efficient $M$-stationarity conditions for the so-called mathematical programs with complementarity constraints (MPCCs), cf. \cite{Out99}. Later, this formula has been adapted to the frequently arising case when the polyhedron $\Gamma$ is given by affine inequalities \cite{HenRoem07}. Meanwhile the researchers started to attack a more difficult case, when $\Gamma$ is the pre-image of a closed set $\Theta$ in a $C^{2}$-mapping $ q $, arising typically in nonlinear or conic programming. It turned ont that one can again obtain an exact formula provided $ \Theta $ is a $C^{2}$-reducible set
(\cite[Definition 3.135]{BonSh00}) and the reference point is nondegenerate with respect to $ q $ and $ \Theta $ (\cite[Definition 4.70]{BonSh00}). In the case of nonlinear programming (NLP) constraints this amounts to the standard Linear independence constraint qualification (LICQ). These results can be found in \cite{MoOut01} and \cite{OutRam11}. The situation, unfortunately, becomes substantially more difficult, provided the nondegeneracy (or LICQ) condition is relaxed. Such a situation has been investigated in the case of strongly amenable $ \Gamma $ in \cite{LeMo04} and \cite{MoOut07} and in the case of NLP constraints under Mangasarian-Fromovitz constraint qualification (MFCQ) in \cite{HenOutSur09}. In both cases one needs to impose still another so-called 2nd-order  qualification condition (SOCQ) to obtain at least an upper estimate of the desired limiting normal cone which is quite often not very tight. By combining results from \cite{HenOutSur09} and \cite{MinSta11} one can further show that in the NLP case the validity of SOCQ is implied by the Constant rank constraint qualification (CRCQ) so that one needs in fact both MFCQ (or its suitable relaxation) and CRCQ \cite{HenOutSur09}. The result of \cite{MoOut07} has been further developed in \cite{MoRo12}, where under a strengthened SOCQ exact formula has been obtained provided the indicatory function of $ \Theta $ is (convex) piecewise linear.

In all above mentioned works the authors employ essentially the generalized differential calculus of B. Mordukhovich as it is presented in \cite{Mo06a} and \cite{RoWe98}. In recent years, however, this calculus has been enriched by H. Gfrerer, who introduced, among other things, a directional variant of the limiting normal cone. This notion has turned out to be very useful in fine analysis of constraint and variational systems, cf. \cite{Gfr13a,Gfr13b,Gfr14a,Gfr14b}.

The aim of the present paper is to compute the limiting normal cone to the graph of $ \hat{N}_{\Gamma} (\cdot)$ with $ \Gamma $ given by NLP constraints under a  {\em different set of assumptions} compared with the above quoted literature. In particular, as in \cite{GfrOut14}, MFCQ is replaced by the metric subregularity of the perturbation mapping at the reference point combined with a uniform metric regularity of this multifunction on a neighborhood, with the reference point excluded. This condition is clearly weaker (less restrictive) than MFCQ. Furthermore, as another ingredient we employ the notion of 2-regularity, introduced in a slightly different context by Avakov \cite{Av85}. This notion enables us to introduce a new CQ called 2-LICQ which ensures an amenable directional behavior of active constraints. On the basis of these two conditions we then compute the directional limiting normal cones (or their upper estimates) to the graph of $ \hat{N}_{\Gamma} $, which eventually  leads to the desired exact formula for the limiting normal cone to the graph of
$ \hat{N}_{\Gamma} $ at the given reference pair.

The plan of the paper is as follows. In Section 2 we collect the needed notions from variational analysis and some essential statements from the literature which are extensively used throughout the whole paper. Furthermore, this section contains a motivating example showing that under mere MFCQ the desired object cannot be generally computed via first and second derivatives of the problem functions. Section 3 is devoted to 2-LICQ. Apart from the definitions one finds there several auxiliary statements needed in the further development. The main results are then collected in Section 4, whereas Section 5 deals with an application of these results to testing of the Aubin property of solution maps to parameterized equilibrium problems, when $ \Gamma $ arises as a constraint set.

Our notation is basically standard. For a cone $K$ with vertex at $0$, $K^\circ$ denotes its negative polar cone, $\Gr F$ stands for the graph of a mapping $F$ and $\B$ signifies the closed unit ball. Finally, $\dist{x,\Omega}$ denotes the distance of the point $x$ to the set $\Omega$.

\section{Background from variational analysis and preliminaries}

 Given a closed set
$\Omega\subset\mathbb R^d$ and a point $\bar z\in\Omega$, define
the (Bouligand-Severi) {\em tangent/contingent cone} to $\Omega$
at $\bar z$  by
\begin{equation}\label{tan}
T_\Omega(\bar z):=\Limsup_{t\downarrow 0}\frac{\Omega-\bar
z}{t}=\Big\{u\in\mathbb R^d\Big|\;\exists\,t_k\downarrow
0,\;u_k\to u\;\mbox{ with }\;\bar z+t_k u_k\in\Omega ~\forall ~ k\}.
\end{equation}
The (Fr\'{e}chet) {\em regular normal cone} to $\Omega$ at $\bar
z\in\Omega$ can be defined by
\begin{equation}\label{fn}
\widehat N_\Omega(\bar
z):=\Big\{v^\ast\in\R^d\Big|\;\limsup_{z\stackrel{\Omega}{\to}\bar
z}\frac{\skalp{ v^\ast,z-\bar z}}{\|z-\bar z\|}\le
0\Big\}
\end{equation}
or equivalently by
\[\widehat N_\Omega(\bar
z):=(T_\Omega(\bar z))^\circ.\]
%%%
The {\em limiting} (Mordukhovich) {\em  normal cone} to $\Omega$ at $\bar{z}\in \Omega$, denoted by $N_{\Omega}(\bar{z})$, is defined by
\begin{equation}\label{eq-203}
N_{\Omega}(\bar{z}):= \Limsup\limits_{z \stackrel{\Omega } {\rightarrow}\bar{z}} \widehat{N}_{\Omega}(z).
\end{equation}
%%%
The above notation "$\Limsup$"
stands for the outer  set limit in the sense of
Painlev\'{e}--Kuratowski, see e.g. \cite[Chapter 4]{RoWe98}. Note
that  the regular normal cone and the limiting normal cone  reduce to
the classical  normal cone of convex analysis,
respectively, when the set $\Omega$ is convex. An interested reader can find enough material about the properties of the above notions e.g. in the monographs \cite{RoWe98}, \cite{Mo06a}.

The following directional version of (\ref{eq-203}) has been introduced in \cite{Gfr14b}. Given a direction $u \in \mathbb{R}^{d}$, the {\em limiting} (Mordukhovich) {\em normal cone} to $\Omega$ in the direction $u$ at $\bar{z}\in \Omega$ is defined by
\[
N_{\Omega}(\bar{z}; u):=\{z^{*} | \exists t_{k}\downarrow 0, u_{k}\rightarrow u, z^{*}_{k}\rightarrow z^{*}: z^{*}_{k}\in \widehat{N}_{\Omega}(\bar{z}+ t_{k}u_{k}) \forall k\}.
\]
A closely related notion to $N_{\Omega}(\bar{z}; u)$ has been defined in \cite{GiMo11}.

Considering next a closed-graph set-valued (in particular, single-valued) mapping $\Psi : \mathbb{R}^{d}\rightrightarrows \mathbb{R}^{s}$, we will describe its local behavior around a point from its graph by means of the following notion.

Given $(\bar{z},\bar{w})\in {\rm gph} \Psi$, the {\em limiting coderivative} of $\Psi$ at $(\bar{z}, \bar{w})$ is the multifunction $D^{*}\Psi(\bar{z}, \bar{w}) : \mathbb{R}^{s}\rightrightarrows\mathbb{R}^{d}$ defined by
\[
D^{*}\Psi (\bar{z}, \bar{w})(w^{*}):= \{z^{*}| (z^{*},-w^{*})\in N_{{\rm gph}\Psi} (\bar{z}, \bar{w})\}, ~w^{*}\in \mathbb{R}^{s}.
\]
%%%
In connection with multifunctions arising in the sequel we will extensively employ  the stability properties defined next.

\begin{definition}Let $\Psi:\R^d\rightrightarrows\R^s$ be a multifunction,  $(\bar u,\bar v)\in\Gr \Psi$ and
 $\kappa>0$. Then
 %%%
\begin{enumerate}
\item $\Psi$ is called {\em metrically regular with modulus
$\kappa$} near $(\bar u,\bar v)$ if there are neighborhoods $U$ of
$\bar u$ and $V$ of $\bar v$ such that

\begin{equation}
\label{EqMetrReg} \dist{u,\Psi^{-1}(v)}\leq
\kappa\dist{v,\Psi(u)}\ \forall (u,v)\in U\times V.
\end{equation}
\item $\Psi$ is called {\em metrically subregular with modulus
$\kappa$} at $(\bar u,\bar v)$ if there is a neighborhood $U$ of
$\bar u$ such that
\begin{equation}
\label{EqMetrSubReg} \dist{u,\Psi^{-1}(\bar v)}\leq
\kappa\dist{\bar v,\Psi(u)}\ \forall u\in U.
\end{equation}
\end{enumerate}
\end{definition}

Consider now the set $\Gamma \subset \mathbb{R}^{m}$ defined by
\begin{equation}\label{eq-207}
\Gamma = \{ y | q_{i}(y)\leq 0, ~i = 1,2,\ldots, l\},
\end{equation}
%%%
where the functions $q_{i}$ are twice continuously differentiable. We could conduct our analysis without much additional effort also for $\Gamma$ given by inequalities and {\em equalities}, but for the sake of brevity we prefer to stick only to (\ref{eq-207}). Note that we do not impose any kind of convexity assumptions. A central object in this paper is the regular normal-cone mapping $\widehat{N}_{\Gamma}(\cdot)$  with $\Gamma$ from (\ref{eq-207}). If the {\em perturbation mapping}
%%%
\begin{equation}\label{eq-210}
M_{q}(y):= q(y)-\mathbb{R}^{l}_{-}
\end{equation}
is metrically subregular at $(y, 0)$, then the regular normal cone $\widehat{N}_{\Gamma}(y)$ can be represented as
\[\widehat{N}_{\Gamma}(y)=\nabla q(y)^T N_{\R_-^l}(q(y))=\{\nabla q(y)^T\lambda\mv \lambda\in\R^l_+,\ q(y)^T\lambda=0\}.\]
Given  elements $y\in \Gamma$ and $y^\ast\in \widehat N_\Gamma(y)$
we define by
\[\Lambda(y,y^\ast):=\{\lambda\in N_{\R_-^l}(q(y))\mv \nabla
q(y)^T\lambda=y^\ast\},\] the set of {\em Lagrange multipliers}
associated with $(y,y^\ast)$. Moreover, with $\I(y):=\{i\in\{1,\ldots,l\}\mv q_i(y)=0\}$ being the index set of active constraints, \[\Tlin(y):=\{v\mv \nabla q_i(y)v\leq 0,\ i\in \I(y)\}\]
and
\[K(y,y^\ast):=T_\Gamma(y)\cap (y^\ast)^\perp\]
stand for the {\em linearized cone} to $\Gamma$ at $y$ and {\em critical cone} to $\Gamma$ at $y$ with respect to
$y^\ast$, respectively. Under metric subregularity of $M_q$ at $(y,0)$ the cones $\Tlin(y)$ and $T_\Gamma(y)$ coincide.

Given index sets $I^+\subset\I\subset \{1,\ldots,l\}$ we write
\[\PI:=\{\mu\in \R^l\mv \mu_i=0, i\not\in \I, \mu_i\geq 0, i\in \I\setminus I^+\}\] and
\[K_{I^+,\I}(y):=\{w\in\R^m\mv \nabla q_i(y)w=0,\ i\in I^+,\ \nabla q_i(y)w\leq0,\ i\in \I\setminus I^+\}.\]
Note that $K_{I^+,\I}(y)^\circ=\nabla q(y)^T\PI$. Finally, for $\lambda\in\R^l_+$ we denote by $I^+(\lambda):=\{i\mv\lambda_i>0\}$ the index set of positive components of $\lambda$.

To simplify the notation, for a given reference pair $(\yb,\yba)$, $\yb\in\Gamma$,
$\yba\in\widehat N_\Gamma(\yb)$, fixed throughout this paper, we will
shortly set $\Ib:=\I(\yb)$, $\Lb:=\Lambda(\yb,\yba)$,
$\Kb:=K(\yb,\yba)$ and $\KbI:=K_{I^+,\I}(\yb)$.

The formulas collected in the next statement have been proved in \cite{MoOut01} and \cite[Chapter 13]{RoWe98}.

\begin{theorem}\label{ThNormalConeLICQ}
Assume that LICQ is fulfilled at $\yb$ and let $\Lb=\{\lb\}$ denote the unique multiplier associated with $(\yb,\yba)$. Then
\[T_{\Gr \widehat N_\Gamma}(\yb,\yba)=\{(v,v^\ast)\mv
v^\ast\in \nabla^2(\lb^Tq)(\yb)v+N_{\Kb}(v)\},\]
\[\widehat N_{\Gr \widehat N_\Gamma}(\yb,\yba)=\{(w^\ast,w)\mv w\in\Kb, w^\ast+\nabla^2(\lb^Tq)(\yb)w\in\Kb^\circ\}\]
and
%%%%
\begin{equation}\label{eq-208}
N_{{\rm gph}\hat{N_{\Gamma}}}(\bar{y}, \bar{y}^{*})=\{(w^{*}, w)| w^{*}+\nabla^{2}(\bar{\lambda}^{T}q)(\bar{y})w \in \nabla q(\bar{y})^{T}D^{*}N_{\mathbb{R}^{l}_{-}}(q(\bar{y}),\bar{\lambda})(-\nabla q(\bar{y})w)\}.
\end{equation}
\end{theorem}

Since the last term on the right-hand side of (\ref{eq-208}) can be expressed in terms of problem data, one also has
%%%
\begin{equation}\label{eq-209}
N_{{\rm gph}\hat{N_{\Gamma}}}(%y, y^{*}
\yb,\yba)= \bigcup_{I^+(\lb)\subset I^+\subset\I\subset\Ib}\{(w^\ast,w)\mv w\in\KbI, w^\ast+\nabla^2(\lb^Tq)(\yb)w\in\KbI^\circ\}.
\end{equation}

If we drop LICQ, a natural option would be to require MFCQ at $\bar{y}$, i.e., the metric regularity of the  perturbation mapping $M_{q}$ given by \eqref{eq-210}
near $(\bar{y}, 0)$. As in \cite{GfrOut14}, however, our work will be based on a  weaker notion.

\begin{definition}\label{DefMetrRegExc}Let $\yb\in\Gamma$. We say that $M_{q}$
is {\em  metrically regular in the vicinity of} $\bar{y}$, if there is some neighborhood $V$ of $\yb$
and some constant $\kappa>0$ such that for every $y\in
M^{-1}(0)\cap V$, $y\not=\yb$, the multifunction $M_{q}$ is metrically
regular near $(y,0)$ with modulus $\kappa$.
\end{definition}
This property is, in particular, implied in the following way:

\begin{definition}\label{DefSOSCMS}
We say that the {\em second order sufficient condition for metric
subregularity} (SOSCMS) holds at $\yb\in\Gamma$, if
for every
$0\not=u\in\Tlin(\yb)$ one has
\[%\nabla q(\yb)^T\lambda=0,\ \lambda\in \widehat N_{\R^l_-}(q(\yb))
\lambda \in \ker (\nabla q(\yb)^T)\cap\widehat N_{\R^l_-}(q(\yb)),
\ u^T\nabla^2 (\lambda^Tq)(\yb)u\geq0\ \Longrightarrow
\lambda=0.\]
\end{definition}
\begin{proposition}[{\cite[Theorem 6.1]{Gfr11}, \cite[Proposition 3]{GfrOut14}}]
Let $\yb\in\Gamma$. Under SOSCMS the mapping $M_q$ is metrically subregular at $(\yb,0)$ and metrically regular in the vicinity of $\bar{y}$.
\end{proposition}
Since MFCQ can be equivalently characterized by the condition
\[\lambda \in \ker (\nabla q(\yb)^T)\cap\widehat N_{\R^l_-}(q(\yb))
 \Longrightarrow
\lambda=0,\]
MFCQ implies SOSCMS.

To present the respective results about $T_{{\rm gph}\hat{N}_{\Gamma}}$ and $\widehat{N}_{{\rm gph}\hat{N}_{\Gamma}}$, we introduce some additional notation.

Given $(y,y^\ast)\in\Gr\widehat N_\Gamma$ we introduce the index set $I^+(y,y^\ast):=\bigcup_{\lambda\in\Lambda(y,y^\ast)}I^+(\lambda)$. With a direction $v\in\Tlin(y)$ let us now associate the index set $\I(y;v):=\{i\in\I(y)\mv \nabla q_i(y)v= 0\}$ and the {\em directional multiplier set} $\Lambda(y,y^\ast;v)$ as the solution set of the linear optimization problem
\begin{equation}\label{EqLPDirMult}
\max_{\lambda\in \Lambda(y,y^\ast)}v^T\nabla^2(\lambda^Tq)(y)v.
\end{equation}
The collection of the extreme points of the polyhedron $\Lambda(y,y^\ast)$ is denoted by $\E(y,y^\ast)$ and we set $\Lambda^\E(y,y^\ast;v):=\Lambda(y,y^\ast;v)\cap\co \E(y,y^\ast)$. Recall that $\lambda\in \Lambda(y,y^\ast)$ is an extreme point of $\Lambda(y,y^\ast)$ if and only if the family $\nabla q_i(y)$, $i\in I^+(\lambda)$, is linearly independent. Since there are only finitely many subsets of $\{1,\ldots,l\}$ it follows that for every $y\in\Gamma$ there is some constant $\kappa$ such that
\begin{equation}
\label{EqBndExtremePoints}\norm{\lambda}\leq\kappa \norm{y^\ast}\ \forall y^\ast\in\R^m\forall\lambda\in\E(y,y^\ast)
\end{equation}
We now define for each $v\in \Null(y):=\{v\in\R^m\mv \nabla q_i(y)v=0,\ i\in\I(y)\}$, i.e. the null space of the gradients of the active inequalities,  the sets
\[{\cal W}(y,y^\ast;v):=\{w\in K(y,y^\ast)\mv w^T\nabla^2((\lambda^1-\lambda^2)^Tq)(y)v=0,\forall
\lambda^1,\lambda^2\in\Lambda(y,y^\ast;v)\},\]
\[\tilde\Lambda^\E(y,y^\ast;v):=\begin{cases}\Lambda^\E(y,y^\ast;v)
&\mbox{if $v\not=0$,}\\
\co(\bigcup\limits_{0\not=u\in K(y,y^\ast)}\Lambda^\E(y,y^\ast; u))&\mbox{if $v=0$,
$K(y,y^\ast)\not=\{0\}$,}\end{cases}\]
 and for each $w\in K(y,y^\ast)$ the set
\begin{eqnarray*}
\lefteqn{L(y,y^\ast;v;w)}\\
&:=&\begin{cases}\{-\nabla^2(\lambda^Tq)(y)w\mv \lambda\in\tilde\Lambda^\E(y,y^\ast;v)\}
+(K(y,y^\ast))^\circ&\mbox{if $K(y,y^\ast)\not=\{0\}$}\\
\R^m&\mbox{if $K(y,y^\ast)=\{0\}$}. \end{cases}
\end{eqnarray*}
Again we will simplify the notation for quantities depending on $\yb$ or $(\yb,\yba)$ by using an overline, i.e., we will write $\Ib(v)$, $\Lb(v)$, $\bar{\cal W}(v)$, etc. instead of  $\I(\yb;v)$, $\Lambda(\yb,\yba;v)$, ${\cal W}(\yb,\yba;v)$ etc.

\begin{theorem}[{\cite[Theorems 1,2]{GfrOut14}}]\label{ThRegNormalCone}Let $(y,y^\ast)\in\Gr\widehat N_\Gamma$ and assume that $M_q$ is metrically subregular at $(y,0)$. Then
\begin{equation}
\label{EqInclTangCone}T_{\Gr \widehat N_\Gamma}(\yb,\yba)\supset\{(v,v^\ast)\mv \exists \lambda\in \Lb(v):
v^\ast\in \nabla^2(\lambda^Tq)(\yb)v+N_{\Kb}(v)\}
\end{equation}
and
\begin{equation}\label{EqInclNormalCone}\widehat N_{\Gr \widehat N_\Gamma}(y,y^\ast)\subset \{(w^\ast,w)\mv w\in
\bigcap_{v\in\Null(y)}{\cal W}(y,y^\ast;v), w^\ast\in
\bigcap_{v\in\Null(y)}L(y,y^\ast;v;w)\}.
\end{equation}
Equality holds in \eqref{EqInclTangCone} if, in addition, $M_q$ is  metrically regular in the vicinity of $y$ and \eqref{EqInclNormalCone} holds with equality if $M_q$ is  metrically regular in the vicinity of $y$ and either  for any
$0\not=v_1,v_2\in K(y,y^\ast)$ it holds $\Lambda^\E(y,y^\ast;v_1)=\Lambda^\E(y,y^\ast;v_2)$ or $I^+(y,y^\ast)=\I(y)$.
\end{theorem}

Very little is known about the limiting normal cone, if we drop the assumption of LICQ.
The following example demonstrates that in general we cannot describe the limiting normal cone $N_{\Gr \widehat N_\Gamma}(\yb,\yb^\ast)$ by first-order and second-order derivatives of $q$ at $\yb$, if the only constraint qualification we assume is  MFCQ.

\begin{example}Let
\[\Gamma:=\left\{y\in\R^3\mv \begin{array}{l}q_1(y):=y_3-y_1^3\leq 0\\q_2(y):=y_3-a^3y_2^3\leq 0
\end{array}\right\},\]
where $a>0$ denotes a fixed parameter, and $(\yb,\yb^\ast)=(0,0)$. Obviously MFCQ is fulfilled at $\yb$. Straightforward calculations yield
\[\widehat N_\Gamma(y)=\begin{cases}\{(-3y_1^2\lambda_1,0,\lambda_1)\mv \lambda_1\geq 0\}&\mbox{if $y_1<ay_2$, $y_3=y_1^3$,}\\
\{(0,-3a^3y_2^2\lambda_2,\lambda_2)\mv \lambda_2\geq 0\}&\mbox{if $y_1>ay_2$, $y_3=a^3y_2^3$,}\\
\{(-3y_1^2\lambda_1,-3a^3y_2^2\lambda_2,\lambda_1+\lambda_2)\mv \lambda_1,\lambda_2\geq 0\}&\mbox{if $y_1=ay_2$, $y_3=y_1^3$,}\\
\{(0,0,0)\}&\mbox{if $y_3<\min\{y_1^3,a^3y_2^3\}$,}\\
\emptyset&\mbox{if $y_3>\min\{y_1^3,a^3y_2^3\}$.}
\end{cases}\]
By applying Theorems \ref{ThNormalConeLICQ},\ref{ThRegNormalCone} we obtain for an arbitrary pair $(y,y^\ast)\in\Gr\widehat N_\Gamma$ that the set $\widehat N_{\Gr\widehat N_\Gamma}(y,y^\ast)$ consists of the collection of all $(w^\ast,w)\in\R^3\times\R^3$ satisfying
\begin{enumerate}
\item $w_3=3y_1^2w_1, w_1^\ast=6\lambda_1y_1w_1-3w_3^\ast y_1^2,w_2^\ast=0$, if $y_1<ay_2$, $y_3=y_1^3$, $y^\ast=(-3y_1^2\lambda_1,0,\lambda_1)$, $\lambda_1>0$,
\item $w_3\leq 3y_1^2w_1,  w_1^\ast=-3w_3^\ast y_1^2, w_2^\ast=0,w_3^\ast\geq 0$, if $y_1<ay_2$, $y_3=y_1^3$, $y^\ast=0$,
\item $w_3=3a^3y_2^2w_2, w_1^\ast=0, w_2^\ast=6\lambda_2a^3y_2w_2-3w_3^\ast a^3y_2^2$, if $y_1>ay_2$, $y_3=a^3y_2^3$, $y^\ast=(0,-3a^3y_2^2\lambda_2,\lambda_2)$, $\lambda_2>0$,
\item $w_3\leq 3a^3y_2^2w_2, w_1^\ast=0, w_2^\ast=-3w_3^\ast a^3y_2^2, w_3^\ast\geq 0$, if $y_1>ay_2$, $y_3=a^3y_2^3$, $y^\ast=0$,
\item $w_3=3y_1^2w_1$, $w_1\leq aw_2$, $w_2^\ast\leq0$, $w_3^\ast=-\frac 1{3y_1^2}(w_1^\ast-6\lambda_1y_1w_1+\frac{w_2^\ast}a)$, if $0\not=y_1=ay_2$, $y_3=y_1^3$, $y^\ast=(-3y_1^2\lambda_1,0,\lambda_1)$, $\lambda_1>0$,
\item $w_3=3a^3y_2^2w_2$, $w_1\geq aw_2$, $w_1^\ast\leq0$,$ w_3^\ast=-\frac1{3a^3y_2^2}(w_2^\ast-6\lambda_2a^3y_2w_2+aw_1^\ast)$,  if $0\not=y_1=ay_2$, $y_3=y_1^3$, $y^\ast=(0,-3a^3y_2^2\lambda_2,\lambda_2)$, $\lambda_2>0$,
\item $w_3=3y_1^2w_1$, $w_1= aw_2$, $w_3^\ast=-\frac1{3y_1^2}(w_1^\ast+\frac{w_2^\ast}a-6y_1w_1(\lambda_1+\lambda_2))$, if $0\not=y_1=ay_2$, $y_3=y_1^3$, $y^\ast=(-3y_1^2\lambda_1,-3a^3y_2^2\lambda_2,\lambda_1+\lambda_2)$, $\lambda_1,\lambda_2>0$,
\item $w_3\leq3\min\{y_1^2w_1,a^3y_2^2w_2\}$, $w_1^\ast,w_2^\ast\leq 0$, $w_3^\ast=-\frac 1{3y_1^2}(w_1^\ast+\frac{w_2^\ast}a)$, if $0\not=y_1=ay_2$, $y_3=y_1^3$, $y^\ast=0$,
\item $w_3=0$, $w_1^\ast=w_2^\ast=0$, if $y=0$, $y^\ast\not=0$,
\item $w_3\leq0$, $w_1^\ast=w_2^\ast=0$, $w_3^\ast\geq0$, if $y=y^\ast=0$,
\item $w^\ast=0$, if $y_3<\min\{y_1^3,a^3y_2^3\}$.
\end{enumerate}
To compute the limiting normal cone $N_{\Gr\widehat N_\Gamma}(\yb,\yba)$, let $(w^\ast,w)\in N_{\Gr\widehat N_\Gamma}(\yb,\yba)$ and consider sequences $(y_k,{y_k}^\ast)\to(\yb,\yba)$, $(w_k^\ast, w_k)\to(w^\ast,w)$ with $(w_k^\ast, w_k)\in\widehat N_{\Gr\widehat N_\Gamma}(y_k,y_k^\ast)$. Then, for infinitely many $k$ the pair $(y_k,y_k^\ast)$ belongs to one of the above subcases and we obtain
\begin{itemize}
\item $w_3=0$, $w_1^\ast=w_2^\ast=0$ in case of 1., 3., 9.,
\item $w_3\leq 0$, $w_1^\ast=w_2^\ast=0$, $w_3^\ast\geq 0$ in case of 2., 4., 8., 10.,
\item $w_3=0$, $w_1\leq aw_2$, $w_2^\ast=-aw_1^\ast\leq 0$ in case of 5.,
\item $w_3=0$, $w_1\geq aw_2$, $w_2^\ast=-aw_1^\ast\geq 0$ in case of 6.,
\item $w_3=0$, $w_1= aw_2$, $w_2^\ast=-aw_1^\ast$ in case of 7.,
\item $w^\ast=0$ in case of 11.
\end{itemize}
We can further conclude that
\begin{eqnarray*}N_{\Gr\widehat N_\Gamma}(\yb,\yba)&=&(\{(0,0,w_3^\ast)\}\times\{(w_1,w_2,0)\}) \cup(\{(0,0,0)\}\times\{(w_1,w_2,w_3)\})\\
&&\cup  (\{(0,0,w_3^\ast)\mv w_3^\ast\geq 0\}\times\{(w_1,w_2,w_3)\mv w_3\leq 0\})\\
&&\cup (\{(w_1^\ast, -aw_1^\ast,w_3^\ast)\mv w_1^\ast\geq0\}\times\{(w_1,w_2,0)\mv w_1\leq aw_2\})\\
&&\cup (\{(w_1^\ast, -aw_1^\ast,w_3^\ast)\mv w_1^\ast\leq0\}\times\{(w_1,w_2,0)\mv w_1\geq aw_2\}).
\end{eqnarray*}
We see that the limiting normal cone depends explicitly on the parameter $a$ as contrasted with the first-order and second-order derivatives of our problem functions $q_i$ at $\yb$. Hence in this situation it is not possible to get a point-based representation of the limiting normal cone by first-order and second-order derivatives.\hfill$\triangle$
\end{example}

\section{2--Regularity and 2-LICQ}
In \cite{Av85}, Avakov introduced the following concept of 2--regularity.
\begin{definition}
Let $g:\R^m\to\R^p$ be twice Fr\'echet differentiable at $\yb
\in\R^m$. We say that $g$ is {\em 2--regular} at the point $\yb$
in a direction $v\in\R^m$, if for all $\alpha\in\R^p$ the system
\begin{equation}\label{Eq2Reg1}
\nabla g(\yb)u+v^T\nabla^2 g(\yb)w=\alpha,\ \nabla g(\yb)w=0.
\end{equation}
has a solution $(u,w)\in\R^m\times\R^m$.
\end{definition}
Note that Avakov \cite{Av85} used this concept only for directions $v$ satisfying $\nabla g(\yb)v=0$, $v^T\nabla^2 g(\yb)v\in \range \nabla g(\yb)$.

Given a direction $v\in\R^n$ and positive scalars $\epsilon,\delta$, the set $V_{\epsilon,\delta}(v)$ is defined by
\[V_{\epsilon,\delta}(v):=\begin{cases}\{0\}\cup\{u\in \epsilon\B\setminus\{0\}\mv \Norm{\frac{u}{\norm{u}}-\frac{v}{\norm{v}}}\leq \delta\}&\mbox{if $v\not=0$,}\\
\epsilon\B&\mbox{if $v=0$.}\end{cases}\]

\begin{proposition}\label{Prop2Reg}
Let $g:\R^m\to\R^p$ be twice Fr\'echet differentiable at $\yb
\in\R^m$ and let $0\not=v\in\R^m$. Then the following statements
are equivalent:
\begin{enumerate}
\item[(a)]$g$ is 2--regular at $\yb$ in direction $v$,
\item[(b)]the implication
\begin{equation}\label{Eq2Reg2}
\nabla g(\yb)^T\lambda = 0,\ (v^T\nabla^2 g(\yb))^T\lambda+\nabla g(\yb)^T\mu=0\quad \Longrightarrow\quad \lambda=0
\end{equation}
holds true,
\item[(c)]there are positive numbers $\epsilon,\delta$ and
$\kappa$ such that for all $(y,z)\in(\yb,g(\yb))+V_{\epsilon,\delta}(v,\nabla g(\yb)v))$ with
$y\not=\yb$ and
$\norm{z-g(y)}\leq \delta\norm{y-\yb}^2$ one has
\[\dist{y,g^{-1}(z)}\leq\frac\kappa{\norm{y-\yb}}\norm{z-g(y)}, \]
\item[(d)]there are positive numbers $\tilde\epsilon$, $\tilde\delta$ and  $\kappa'$ such that for all $y\in \yb+  V_{\tilde\epsilon,\tilde\delta}(v)$ one has
\[\inf_{0\not=\lambda\in\R^p}\frac{\norm{\nabla g(y)^T\lambda}}{\norm{\lambda}}\geq \frac {\norm{y-\yb}}{\kappa'}.\]
\end{enumerate}
\end{proposition}

\begin{proof}
The equivalence $(a)\Leftrightarrow(b)$ is an immediate consequence of the fundamental theorem of linear algebra, which states in particular that for every matrix $A$ the kernel $\ker A$ is the orthogonal complement of the row space $\range(A^T)$. Hence, $g$ is 2--regular at $\yb$ in direction $v$, if and only if
\[\R^p\times\{0\}^p\subset\range\left(\begin{matrix}v^T\nabla^2 g(\yb)&\nabla g(\yb)\\\nabla g(\yb)&0\end{matrix}\right)=\left(\ker\left(\begin{matrix}(v^T\nabla^2 g(\yb))^T&\nabla g(\yb)^T\\\nabla g(\yb)^T&0\end{matrix}\right)\right)^\perp,\]
being equivalent to
\[\{0\}^p\times\R^p\supset\ker\left(\begin{matrix}(v^T\nabla^2 g(\yb))^T&\nabla g(\yb)^T\\\nabla g(\yb)^T&0\end{matrix}\right)\]
which is exactly \eqref{Eq2Reg2}.
Note that by \cite[Definition 1]{Gfr14a} statement $(c)$ is nothing else than the statement that the multifunction $\Psi(y):=\{g(y)\}$ is metrically pseudo-regular of order 2 in direction $(v,\nabla g(\yb)v)$ at $(\yb,g(\yb))$ and the equivalence $(b)\Leftrightarrow(c)$ has already been established in \cite[Theorem 2, Remark 5]{Gfr14a}.
 Next we show the implication $(c)\Rightarrow(d)$. By \cite[Lemma 1]{Gfr14a}, condition $(c)$ implies that there are $\epsilon',\delta',\kappa'>0$ such that for every $\tilde y\not=\yb$ with $(\tilde y,g(\tilde y))\in (\yb,g(\yb))+V_{\epsilon',\delta'}(v,\nabla g(\yb)v)$
 the multifunction $\Psi$ is metrically regular near $(\tilde y,g(\tilde y))$ with modulus $\kappa'/\norm{\tilde y-\yb}$. By using the inequality
 \begin{equation}\label{EqDistDir}\Norm{\frac{u}{\norm{u}}-\frac{u'}{\norm{u'}}}\leq
2\frac{\norm{u-u'}}{\max\{\norm{u},\norm{u'}\}}\end{equation}
 with $u=(\tilde y-\yb,g(\tilde y)-\yb)$ and $u'=\frac{\norm{\tilde y-\yb}}{\norm{v}}(v,\nabla g(\yb)v)$ and, by taking into account that
 $g(\tilde y)-g(\yb)=\norm{\tilde y-\yb}(\nabla g(\yb)\frac{v}{\norm{v}}+\nabla g(\yb)(\frac{\tilde y-\yb}{\norm{\tilde y-\yb}}-\frac{v}{\norm{v}}))+\oo(\norm{\tilde y-\yb}$, we obtain
\begin{eqnarray*}\lefteqn{\Norm{\frac{(\tilde y-\yb,g(\tilde y)-g(\yb))}{\norm{(\tilde y-\yb,g(\tilde y)-g(\yb))}}-\frac {(v,\nabla g(\yb)v)}{\norm{(v,\nabla g(\yb)v)}}}}\\
&\leq& 2\frac{\norm{u-u'}}{\norm{\tilde y-\yb}}
=2\Norm{\left(\frac{\tilde y-\yb}{\norm{\tilde y-\yb}}-\frac v{\norm{v}},
\nabla g(\yb)(\frac{\tilde y-\yb}{\norm{\tilde y-\yb}}-\frac v{\norm{v}})+\frac{\oo(\norm{\tilde y-\yb})}{\norm{\tilde y-\yb}}\right)}.\end{eqnarray*}
Hence we can choose $\tilde \epsilon>0$ and $\tilde\delta>0$ small enough, such that for all $\tilde y\in\yb+ V_{\tilde\epsilon,\tilde\delta}(v)$ we have $(\tilde y,g(\tilde y))\in (\yb,g(\yb))+ V_{\epsilon',\delta'}(v,\nabla g(\yb)v)$.  Now statement (d) follows from \cite[Example 9.44]{RoWe98}. Finally, we prove the implication $(d)\Rightarrow(b)$ by contraposition. Assuming that there are $0\not=\bar\lambda\in\R^p$, $\bar\mu\in\R^p$ with $\nabla g(\yb)^T\bar\lambda = 0$ and $(v^T\nabla^2 g(\yb))^T\bar\lambda+\nabla g(\yb)^T\bar\mu=0$, we have
\[\nabla g(\yb+tv)^T(\bar\lambda+t\bar\mu)=\nabla g(\yb)^T\bar\lambda +t((v^T\nabla^2 g(\yb))^T\bar\lambda+\nabla g(\yb)^T\bar\mu)+\oo(t)=\oo(t)\]
and therefore
\[\inf_{0\not=\lambda\in\R^p}\frac{\norm{\nabla g(\yb+tv)^T\lambda}}{\norm{\lambda}}=\oo(t)\]
contradicting $(d)$.
\end{proof}
\begin{remark}\label{Rem2Reg}Statement (d) of Proposition \ref{Prop2Reg} says that for every $y\in \yb+  V_{\tilde\epsilon,\tilde\delta}(v)$ with $y\not=\yb$ the Jacobian $\nabla g(y)$ has full rank and its smallest singular value is bounded below by $\norm{y-\yb}/\kappa'$. Consequently, for every right hand side $\alpha\in\R^p$ the system $\nabla g(y)u=\alpha$ has a solution $u$ satisfying
\[\norm{u}\leq \frac{\kappa'\norm{\alpha}}{\norm{y-\yb}}.\]
\end{remark}

The following lemma is useful for estimating index sets of active
constraints:
\begin{lemma}\label{LemEstActConstr}
Let $g:\R^m\to\R^p$ be twice Fr\'echet differentiable  at
$\yb\in\R^m$, $g(\yb)=0$, let $I\subset\{1,\ldots,p\}$ and let
$v\in\R^m$ with $\nabla g(\yb)v=0$ be given. Then there are
sequences $(t_k)\downarrow 0$, $(v_k)\to v$ such that
\begin{equation}\label{EqActConstr1}\lim_{k\to\infty}t_k^{-2}g_i(\yb+t_kv_k)=0,\ i\in I, \limsup_{k\to\infty} t_k^{-2}g_i(\yb+t_kv_k)\leq 0,\ i\in \{1,\ldots,p\}\setminus I\end{equation}
if and only if there is some $\bar z\in\R^m$ with
\begin{equation}\label{EqActConstr2}\nabla g_i(\yb)\bar z+v^T\nabla^2 g_i(\yb)v\begin{cases}=0&\mbox{if $i\in I$}\\\leq 0&\mbox{if $i\in\{1,\ldots,p\}\setminus I$.}\end{cases}\end{equation}
\end{lemma}
\begin{proof}
To show the ''only if'' part, let $(t_k)\downarrow 0$ and
$(v_k)\to v$ be given, such that \eqref{EqActConstr1} holds and
consider for every $b\in\R^p$ the set
\[\Delta(b):=\left\{z\in\R^m\mv \nabla g_i(\yb)z+b_i=0,\ i\in  I,\
\nabla g_i(\yb)z+b_i\leq0,\ i\in \{1,\ldots,p\}\setminus
I\right\}.\] By Hoffman's Lemma there is some constant
$\tilde\beta$ such that for all $z\in\R^m$ and all $b$ with
$\Delta(b)\not=\emptyset$  we have
\[\dist{z,\Delta(b)}\leq\tilde\beta(\sum_{i\in I}\vert \nabla g_i(\yb)z+b_i\vert +\sum_{i\in\{1,\ldots,p\}\setminus I}\max\{\nabla g_i(\yb)z+b_i,0\}).\]
For every $k$ let $r^k:=2g(\yb+t_kv_k)/t_k^2-(2\nabla
g(\yb)(v_k-v)/t_k+ v^T\nabla^2 g(\yb)v)$. Because of $(v_k)\to v$, $g(\yb)=0$ and $\nabla g(\yb)v=0$
we have
\[0=\lim_{k\to\infty}\frac{g(\yb +t_kv_k)-(g(\yb)+t_k\nabla g(\yb)v_k+\frac 12 t_k^2 v^T\nabla^2g(\yb)v)}{t_k^{2}/2}=\lim_{k\to\infty}r^k.\]
Setting
\[b^k_i:=\begin{cases}v^T\nabla^2 g_i(\yb)v +r_i^k-2g_i(\yb+t_kv_k)/t_k^2&\mbox{if $i\in I$}\\
v^T\nabla^2 g_i(\yb)v
+r_i^k-2\max\{g_i(\yb+t_kv_k),0\}/t_k^2&\mbox{if
$i\in\{1,\ldots,p\}\setminus I$,}\end{cases}\] we have
$2(v_k-v)/t_k\in \Delta(b^k)$ and therefore there is some
$z_k\in\Delta(b^k)$ satisfying
\[\norm{z_k}=\dist{0,\Delta(b^k)}\leq\tilde\beta(\sum_{i\in I}\vert b_i^k\vert +\sum_{i\in\{1,\ldots,p\}\setminus I}\max\{b_i^k,0\}).\]
Because of \eqref{EqActConstr1} and $(r^k)\to 0$ we have $(b^k)\to
v^T\nabla^2 g(\yb)v$. Hence the sequence $(z_k)$ is uniformly
bounded and, by eventually passing to a subsequence, $(z_k)$ is
convergent to some $\bar z$. Then we also have $\bar
z\in\Delta(v^T\nabla^2 g(\yb)v)$ and therefore $\bar z$ fulfills
\eqref{EqActConstr2}.

The ''if'' part follows immediately from the observation that,
for every  $\bar z\in\R^m$,  we have
\begin{eqnarray*}\lim_{t\downarrow 0}t^{-2}g(\yb+tv+\frac 12 t^2\bar z)&=&\lim_{t\downarrow 0}t^{-2}\left(g(\yb)+t\nabla g(\yb)v+\frac 12 t^2(\nabla g(\yb)\bar z+v^T\nabla^2 g(\yb)v)\right)\\
&=&\frac 12(\nabla g(\yb)\bar z+v^T\nabla^2 g(\yb)v)
\end{eqnarray*}
due to \cite[Theorem 13.2]{RoWe98}.
\end{proof}
The notion defined below represents a crucial CQ, needed in all our main results.
\begin{definition}\label{Def2LICQ}
Let $v\in\Tlin(\yb)$. We say that {\em 2-LICQ} holds at $\yb$ in direction $v$ for the constraints $q_i(y)\leq 0$, $i=1,\ldots,l$, if there are positive numbers $\epsilon,\delta$, such that for every $y\in(\yb+V_{\epsilon,\delta}(v))\cap\Gamma$, $y\not=\yb$, the mapping $(q_i)_{i\in\I(y)}$ is 2-regular at $\yb$ in direction $v$.
\end{definition}

We now present  a second-order sufficient condition for 2-LICQ.
\if{
We denote by $Z(y,y^\ast;v)$ the solution set of the linear program
\begin{equation}\label{EqDPDirMult}
\min_z-{y^\ast}^Tz\quad\mbox{subject to}\quad \nabla q_i(y)z+v^T\nabla^2 q_i(y)v\leq 0,\ i\in\I(y),
\end{equation}
which is the dual program to \eqref{EqLPDirMult},
and we denote  by $\Xi(y,y^\ast;v)$ the feasible region of \eqref{EqDPDirMult}. Further we define
\begin{eqnarray*}\lefteqn{\J(y,y^\ast;v)}\\
&:=&\left\{\J\subset\I(y;v)\mv \exists z\in\Xi(y,y^\ast;v): \J=\{i\in\I(y;v)\mv \nabla q_i(y)z+v^T\nabla^2 q_i(y)v=0\}\right\}.
\end{eqnarray*}
}\fi
We denote by $\Zb(v)$ the solution set of the linear program
\begin{equation}\label{EqDPDirMult}
\min_z-{\yba}^Tz\quad\mbox{subject to}\quad \nabla q_i(\yb)z+v^T\nabla^2 q_i(\yb)v\leq 0,\ i\in\Ib,
\end{equation}
which is the dual program to \eqref{EqLPDirMult} at $(\yb,\yba)$,
and we denote  by $\bar\Xi(v)$ the feasible region of \eqref{EqDPDirMult}. Take $z\in \bar\Xi(v)$ and define the following index subset
\[\J(z):= \{i\in\Ib(v)\mv \nabla q_i(\yb)z+v^T\nabla^2 q_i(\yb)v=0\}.\]
Consider now the collection of index subsets $\Jb(v):=\{\J(z)\mv z\in \bar\Xi(v)\}$. In what follows we say that an index set $\hat \J\in \Jb(v)$ is  {\em maximal }, if it is  maximal with respect to the inclusion order, i.e. for any index set $\J\in \Jb(v)$ such that $\hat \J\subset J$ we have $\hat \J=\J$. Note that for each element $\J\in\Jb(v)$ we can always find a maximal element $\hat\J$ of $\Jb(v)$ such that $\J\subset\hat \J$.

\begin{proposition}\label{PropSuffCond2LICQ}
Let $v\in \Tlin(\yb)$  and assume that for every maximal index set $\hat\J\in \Jb(v)$ the mapping $(q_i)_{i\in\hat\J}$ is 2-regular at $\yb$ in direction $v$. Then 2-LICQ holds at $\yb$ in direction $v$.
\end{proposition}
\begin{proof}
By contraposition. Assuming on the contrary that 2-LICQ does not hold at $\yb$ in direction $v$, there are sequences $(t_k)\downarrow 0$, $(v_k)\to v$ such that
$(q_i)_{i\in\I(y_k)}$ is not 2-regular at $\yb$ in direction $v$, where $y_k:=\yb+t_kv_k\not=\yb$. By passing to a subsequence we can assume that $\I(y_k)=\tilde\I$ holds for all $k$. It follows that
\[\nabla q_i(\yb)v=\lim_{k\to\infty}\frac {q_i(y_k)-q_i(\yb)}{t_k}\begin{cases}=0,&i\in \tilde\I\\\leq 0,&i\in\Ib\setminus\tilde \I\end{cases}\]
showing $\tilde \I\subset\Ib(v)$, and by using Lemma \ref{LemEstActConstr}, there is some $z$ satisfying
\[\nabla q_i(\yb)z+v^T\nabla^2q_i(\yb)v\begin{cases}=0,&i\in\tilde \I\\\leq 0,&i\in\Ib(v)\setminus\tilde \I.\end{cases}\]
Putting $\bar z=z+\alpha v$ for $\alpha$ sufficiently large, we obtain
\[\nabla q_i(\yb)\bar z+v^T\nabla^2q_i(\yb)v\begin{cases}=0,&i\in\tilde \I\\\leq 0,&i\in\Ib\setminus\tilde \I,\end{cases}\]
showing
%%%%%%%%%%%%%
%that there is an index set $\hat\J\in\Jb(v)$ with $\red{\tilde \I}\subset \hat\J$. Since we can choose $\hat J$ as a maximal index set,
%%%%%%%%%%%%
$\tilde \I\subset \J(\bar z)\in \Jb(v)$. Choosing $\hat J$ as a maximal index set with $\J(\bar z)\subset \hat J$, the mapping $(q_i)_{i\in\hat \J}$ is 2-regular at $\yb$ in direction $v$ and we can conclude that $(q_i)_{i\in\tilde \I}$ is 2-regular at $\yb$ in direction $v$, a contradiction.
\end{proof}

\begin{proposition}\label{Prop2LICQ}
Let $v\in \Tlin(\yb)$  and a maximal index set $\hat\J\in \Jb(v)$
be given and assume that $(q_i)_{i\in\hat\J}$ is 2--regular in direction $v$ at $\yb$. Then for every subset $\J\subset\hat\J$ there exists some $\bar\tau>0$ and a mapping $\hat y:[0,\bar\tau]\to \Gamma$ such that $\hat y(0)=\yb$, $\I(\hat y(\tau))=\J$, LICQ is fulfilled at $\hat y(\tau)$ for every $\tau\in (0,\bar\tau)$  and
\[\lim_{\tau\downarrow 0}\frac {\hat y(\tau)-\yb}\tau=v.\]
\end{proposition}
\begin{proof}Let $\J\subset\hat\J$  be arbitrarily fixed and consider %the
an element $\hat z\in \bar\Xi(v)$ with
\[\hat \J=\{i\in \Ib(v)\mv \nabla q_i(\yb)\hat z+  v^T\nabla^2 q_i(\yb) v=0\}.\]
Since $(q_i)_{i\in \hat\J}$ is assumed to be 2--regular in
direction $v$ and
\begin{eqnarray*}q_i(\yb+\tau v+\frac 12 \tau^2\hat
z)&=&q_i(\yb)+\tau\nabla q_i(\yb) v+\frac 12\tau^2(\nabla q_i(\yb)\hat
z+ v^T\nabla^2q_i(\yb) v)+\oo(\tau^2)\\
&=&\oo(\tau^2), i\in\hat \J,\end{eqnarray*}
by means of Proposition \ref{Prop2Reg}(c), we can find for every sufficiently small $\tau>0$
some $\hat y(\tau)$ satisfying  $q_i(\hat y(\tau))=0, i\in\J$, $q_i(\hat y(\tau))= -\tau^4$, $i\in \hat \J\setminus\J$ and
\[\norm{\hat y(\tau)-(\yb+\tau  v+\frac 12 \tau^2\hat z)}\leq \frac{\kappa\norm{r(\tau)}}{\norm{\tau v+\frac 12 \tau^2\hat z}}=\oo(\tau),\]
where \[r_i(\tau):=\begin{cases}q_i(\yb+\tau v+\frac 12\tau^2\hat z)& i\in\J,\\q_i(\yb+\tau v+\frac 12\tau^2\hat z)+\tau^4& i\in\hat\J \setminus\J.\end{cases}\]
We will now show by contraposition that there
is some constant $c>0$ such that $q_i( \hat y(\tau))<-c\tau^2$, $i\in\Ib(v)\setminus \hat\J$, for all $\tau>0$ sufficiently small. Assume on the contrary
that there is an index $j\in\Ib(v)\setminus \hat\J$ and a
sequence $(\tau_k)\downarrow 0$ such that
$\liminf_{k\to\infty}\tau_k^{-2}q_j(\hat y(\tau_k))\geq 0$. Applying Lemma
\ref{LemEstActConstr} to the mapping $(g_i)_{i\in \hat\J\cup\{j\}}$ given by $g_i=q_i$, $i\in\hat\J$ and $g_j=-q_j$, we can
find some $z$ with $\nabla q_i(\yb)z+v^T\nabla^2q_i(\yb) v=0$, $i\in\hat\J$, and $\nabla q_j(\yb)z+v^T\nabla^2q_j(\yb) v\geq0$. The number \[\alpha:=\max\{\alpha\in
[0,1]\mv \nabla q_i(\yb)((1-\alpha)\hat z+\alpha z)+ v^T\nabla^2q_i(\yb) v\leq0, i\in\Ib(v)\setminus \hat\J\}\] is positive because of  $\nabla q_i(\yb)\hat
z+ v^T\nabla^2q_i(\yb)v<0$, $i\in\Ib(v)\setminus\hat\J$.
Thus $z_\alpha:=(1-\alpha)\hat z+\alpha z\in \bar\Xi(v)$, but by
construction, the index set $\hat\J$ is strictly contained in
$\{i\in\Ib(v)\mv \nabla q_i(\yb)
z_\alpha+ v^T\nabla^2q_i(\yb) v=0\}$ contradicting the maximality of
$\hat\J$. Therefore our claim is proved. Since we also have
$q_i(\hat y(\tau))<\nabla q_i(\yb) v/2<0$, $i\in\Ib\setminus \Ib(v)$, and $q_i(\hat y(\tau))<q_i(\yb) /2<0$, $i\not\in\Ib$, for all
$\tau>0$ sufficiently small, we see that $\hat y(\tau)\in\Gamma$ and the
constraints active at $\hat y(\tau)$ are exactly those given by $\J$.
Further, our assumption of 2--regularity ensures that LICQ is
fulfilled at $\hat y(\tau)$, cf. Remark \ref{Rem2Reg}, and this completes the proof.
\end{proof}

\section{Computation of the limiting normal cone}
By the definitions we have the representation
\[N_{\Gr\widehat N_\Gamma}(\yb,\yba)=\widehat N_{\Gr\widehat N_\Gamma}(\yb,\yba)\cup\bigcup_{(v,v^\ast)\not=0}N_{\Gr\widehat N_\Gamma}((\yb,\yba);(v,v^\ast)).\]
We split the calculation of the limiting normal cone in directions of the form $(0,v^\ast)$ into two parts:
\[N_{\Gr\widehat N_\Gamma}((\yb,\yba);(0,v^\ast))=N_{\Gr\widehat N_\Gamma}^1((\yb,\yba);(0,v^\ast))\cup N_{\Gr\widehat N_\Gamma}^2((\yb,\yba);(0,v^\ast)),\]
where
\begin{enumerate}
\item  $N_{\Gr\widehat N_\Gamma}^1((\yb,\yba);(0,v^\ast))$ is the collection of all $(w^\ast,w)$ such that there are sequences $(t_k)\downarrow 0$, $(v_k,v_k^\ast)\to(0,v^\ast)$ and $(w_k^\ast,w_k)\to (w^\ast,w)$ with $v_k\not=0$ and $(w_k^\ast,w_k)\in\widehat N_{\Gr\widehat N_\Gamma}(\yb+t_kv_k,\yba+t_kv_k^\ast)$, and
\item $N_{\Gr\widehat N_\Gamma}^2((\yb,\yba);(0,v^\ast))$ is the collection of all $(w^\ast,w)$ such that there are sequences $(t_k)\downarrow 0$, $(v_k^\ast)\to v^\ast$ and $(w_k^\ast,w_k)\to (w^\ast,w)$ with $(w_k^\ast,w_k)\in\widehat N_{\Gr\widehat N_\Gamma}(\yb,\yba+t_kv_k^\ast)$.
\end{enumerate}

In what follows we use the following notation:
\[\Mb(v,v^\ast):=\{(\lambda,\mu)\in\Lb(v)\times T_{ N_{\R^l_-}(q(\yb))}(\lambda)\mv v^\ast= \nabla^2(\lambda^Tq)(\yb)v+\nabla q(\yb)^T\mu\},\]
\[\KbI(v):=\left\{w\in \KbI\mv \exists z\in\R^m: \nabla q_i(\yb)z+v^T\nabla^2q_i(\yb)w\begin{cases}=0&\mbox{if $i\in I^+$,}\\
\leq 0&\mbox{if $i\in \I\setminus I^+$}\end{cases}\right\},\]
\[Q(v,\lambda,I^+,\I):=\{(w^\ast,w)\mv w\in \KbI(v),\  w^\ast+\nabla^2(\lambda^Tq)(\yb)w\in (\KbI(v))^\circ\},\]
where $I^+\subset\I$ are arbitrary subsets of $\Ib$. Further, for every $(\lambda,\mu)\in\Mb(v,v^\ast)$, we set
\[I^+(\lambda,\mu)=I^+(\lambda)\cup \{i\mv \lambda_i=0, \mu_i>0\}.\]

\begin{lemma}\label{LemDualCone}One has
\[(\KbI(v))^\circ=\{\nabla q(\yb)^T\mu+\nabla^2(\nu^Tq)(\yb)v\mv \mu,\nu\in \PI,\ \nabla q(\yb)^T\nu=0\}.\]
\end{lemma}
\begin{proof}
We have
\[\KbI(v)=\left\{w\mv v^T\nabla^2 q(\yb)w\in \range \nabla q(\yb)+\PI^\circ\right\}\cap \KbI\]
and therefore $(\KbI(v))^\circ=\cl\left(\left\{w\mv v^T\nabla^2 q(\yb)w\in \range \nabla q(\yb)+\PI^\circ\right\}^\circ + \KbI^\circ\right)$.
Since $\range \nabla q(\yb)+\PI^\circ$, $\left\{w\mv v^T\nabla^2 q(\yb)w\in \range \nabla q(\yb)+\PI^\circ\right\}$, $\KbI^\circ$ are convex polyhedral cones and hence so are also their polar cones, we obtain
\begin{eqnarray*}
(\KbI(v))^\circ&=& \left\{w\mv v^T\nabla^2 q(\yb)w\in \range \nabla q(\yb)+\PI^\circ\right\}^\circ + \KbI^\circ\\
&=&(v^T\nabla^2 q(\yb))^T(\range \nabla q(\yb)+\PI^\circ)^\circ+\nabla q(\yb)^T\PI\\
&=&(v^T\nabla^2 q(\yb))^T(\ker \nabla q(\yb)^T\cap\PI)+\nabla q(\yb)^T\PI
\end{eqnarray*}
and the claimed result follows.
\end{proof}

\begin{lemma}\label{LemAux1}
Consider convergent sequences $(t_k)\downarrow 0$, $(v_k,v_k^\ast)\to(v,v^\ast)$, $(\lambda^k)\to\tilde\lambda$ and an index set $I^+$ such that $\lambda^k\in\Lambda(y_k,y_k^\ast)$ and $I^+(\lambda^k)=I^+$ for all $k$, where $(y_k,y_k^\ast):=(\yb,\yba)+t_k(v_k,v_k^\ast)$. Then $\tilde \lambda\in\Lb$ and there is some $\tilde\mu$ such that $(\tilde\lambda,\tilde\mu)\in\Mb(v,v^\ast)$ and $I^+(\tilde\lambda,\tilde\mu)\subset I^+$.
\end{lemma}
\begin{proof}
Obviously we have $\tilde\lambda\in\Lb$ and $I^+(\tilde\lambda)\subset I^+\subset\Ib$. Now consider for every $u^\ast\in\R^m$ the set \[\Delta(u^\ast):=\{\lambda\in\R^l_+\mv\nabla q(\yb)^T\lambda=u^\ast, \lambda_i=0, i\not\in I^+\}.\]
By Hoffman's error bound there is some constant $\beta$ such that for every $u^\ast$ with $\Delta(u^\ast)\not=\emptyset$ and every $\lambda\in\R^l$ one has
\[\dist{\lambda,\Delta(u^\ast)}\leq\beta \left(\norm{\nabla q(\yb)^T\lambda-u^\ast}+\sum_{i\not\in I^+}\vert\lambda_i\vert +\sum_{i\in I^+}\max\{-\lambda_i,0\}\right).\]
Since $\tilde\lambda\in\Delta(\yba)$, for every $k$ there is some $\tilde\lambda^k\in\Delta(\yba)$ satisfying
\begin{eqnarray*}
\norm{\tilde\lambda^k-\lambda^k}&\leq&\beta\norm{\nabla q(\yb)\lambda^k-\yba}=\beta\norm{(\nabla q(\yb)-\nabla q(y_k))^T\lambda^k+t_kv_k^\ast}\\
&=&\beta (t_k\norm{v_k^\ast-\nabla^2(-{\lambda^k}^Tq)(\yb)v_k}+\oo(t_k)),
\end{eqnarray*}
showing that the sequence $\mu^k:=(\lambda^k-\tilde\lambda^k)/t_k$ is bounded. By passing to a subsequence if necessary we can assume that the sequence $(\mu^k)$ converges to some $\hat\mu$. If $\hat\mu\in T_{N_{R^l_-}(q(\yb))}(\tilde\lambda)$, we can take $\tilde\mu=\hat\mu$. Otherwise the index set $L:=\{i\in I^+\setminus I^+(\tilde\lambda)\mv \hat\mu_i<0\}$ is not empty and we fix some index $\bar k$ such that $\mu^{\bar k}_i<\hat\mu_i/2$ $\forall i\in L$ and set $\tilde\mu:=\hat\mu+2(\tilde\lambda^{\bar k}-\tilde\lambda)/t_{\bar k}$.
Then for all $i$
with $\tilde \lambda_i=0$ we have $\tilde\mu_i\geq \hat\mu_i$ and for
all $i\in L$ we have
\[\tilde\mu_i =\hat\mu_i+2(\tilde \lambda^{\bar k}_i-\tilde \lambda_i)/t_{\bar k}\geq \hat\mu_i+2(\tilde \lambda^{\bar k}_i- \lambda_i^{\bar k})/t_{\bar k}\geq0\]
and therefore $\tilde\mu\in T_{N_{\R^l_-}(q(\yb))}(\tilde\lambda)$. Taking into account that $\tilde\lambda\in \Lb$, $\tilde\lambda^{\bar k}\in\Delta(\yba)\subset\Lb$ and thus $\nabla q(\yb)^T\hat\mu=\nabla q(\yb)^T\tilde\mu$, we obtain
\begin{eqnarray*}\nabla q(\yb)^T\tilde\mu&=&\lim_{k\to\infty}\frac{\nabla q(\yb)^T(\lambda^k-\tilde\lambda^k)}{t_k}=\lim_{k^\to\infty}\frac{y_k^\ast+(\nabla q(\yb)-\nabla q(y_k))^T\lambda^k-\yba}{t_k}\\
&=&v^\ast-\nabla^2(\tilde\lambda^Tq)(\yb)v
\end{eqnarray*}
showing $(\tilde\lambda,\tilde\mu)\in\Mb(v,v^\ast)$. By the construction of  $\tilde \mu$ it is clear that $I^+(\tilde\lambda,\tilde\mu)\subset I^+$ and this finishes the proof.
\end{proof}

On the basis of these auxiliary results we may now state the first of the main results of this paper. Note that for the calculation of the directional limiting normal cone we only have to take into account directions $(v,v^\ast)\in T_{\Gr \widehat N_\Gamma}(\yb,\yba)$ because of $N_{\Gr \widehat N_\Gamma}((\yb,\yba);(v,v^\ast))=\emptyset$ whenever $(v,v^\ast)\not\in T_{\Gr \widehat N_\Gamma}(\yb,\yba)$.

\begin{theorem}\label{ThLimNormalCone1}Let $0\not=(v,v^\ast)\in T_{\Gr \widehat N_\Gamma}(\yb,\yba)$ and
assume that $M_q$ is metrically subregular at $\yb$ and  metrically regular in the vicinity of $\yb$.
\begin{enumerate}
\item If $v\not=0$, assume that 2-LICQ holds at $\yb$ in direction $v$. Then
\begin{equation}\label{EqInclLimNormalCone1}N_{\Gr \widehat N_\Gamma}((\yb,\yba);(v,v^\ast))\subset\bigcup\limits_{\AT{(\lambda,\mu)\in\Mb(v,v^\ast)}{\J\in\Jb(v)}}
\bigcup\limits_{I^+(\lambda,\mu)\subset I^+\subset \I\subset\J} Q(v,\lambda,I^+,\I)\end{equation}
and this inclusion holds with equality if for every maximal index set $\J\in \Jb(v)$
the mapping $y\to (q_i(y))_{i\in \J}$ is 2--regular at $\yb$ in direction $v$.
\item If $v=0$, assume that 2-LICQ holds at $\yb$ in every direction $0\not=u\in \Kb$.  Then
\begin{eqnarray}\nonumber\lefteqn{N_{\Gr \widehat N_\Gamma}^1((\yb,\yba);(0,v^\ast))}\\
\label{EqInclLimNormalCone2}&\subset&
\bigcup\limits_{\AT{\tilde v\in \Kb}{\norm{\tilde v}=1}}
\bigcup\limits_{\AT{(\lambda,\mu)\in\Mb(0,v^\ast):\lambda\in\Lb(\tilde v)}{\J\in\Jb(\tilde v)}}\bigcup\limits_{I^+(\lambda,\mu)\subset I^+\subset \I\subset\J} Q(\tilde v,\lambda,I^+,\I).\end{eqnarray}
Now equality holds if for every direction $0\not=u\in \Kb$ and every maximal index set $\J\in\Jb(u)$
 the mapping $y\to (q_i(y))_{i\in \J}$ is 2--regular at $\yb$ in direction $u$.
\end{enumerate}
\end{theorem}
\begin{proof}
In the first part of the proof we show the inclusions \eqref{EqInclLimNormalCone1} and \eqref{EqInclLimNormalCone2}, respectively. Consider $(\wb^\ast,\wb)\in N_{\Gr\widehat N_\Gamma}((\yb,\yba);(v,v^\ast))$ if $v\not=0$, and $(\wb^\ast,\wb)\in N_{\Gr\widehat N_\Gamma}^1((\yb,\yba);(0,v^\ast))$ if $v=0$, respectively. Then there are sequences $(w_k^\ast, w_k)\to (\wb^\ast,\wb)$, $(t_k)\downarrow 0$,
$(v_k,v_k^\ast)\to(v,v^\ast)$ such that $v_k\not=0$ and $(w_k^\ast, w_k)\in\widehat N_{\Gr\widehat N_\Gamma}(y_k,y_k^\ast)$
where $(y_k,y_k^\ast):=(\yb,\yba)+t_k(v_k,v_k^\ast)$. Next we define $\tilde v_k:=v_k/\norm{v_k}$, $\tilde t_k:=t_k\norm{v_k}$,
if $v=0$, and $\tilde v_k:=v_k$, $\tilde t_k:=t_k$, if $v\not=0$.
By eventually passing to some subsequence in case $v=0$, we can assume that $\tilde v_k$ converges to some $\tilde v$ and we will now show that there are
multipliers  $(\tilde \lambda,\tilde\mu)\in\Mb(v,v^\ast)$ with $\tilde\lambda\in\Lb(\tilde v)$ and index sets $\tilde  I^+$,$\tilde \I$, $\J$ with
 $I^+(\tilde\lambda,\tilde\mu)\subset\tilde I^+\subset\tilde \I\subset\J\in\Jb(\tilde v)$ such  that $\wb\in \KbIt(\tilde v)$, $\wb^\ast+\nabla^2(\tilde\lambda^Tq)(\yb)\wb\in (\KbIt(\tilde v))^\circ$.

Since $y_k\not=\yb$, as a
consequence of the assumption that $M_q$ is  metrically regular in the vicinity of $\yb$, with each $y_k^\ast$ there is associated some multiplier
$\lambda^k\in N_{N_{\R^l_-}(q(\yb))}(q(y_k))$ with
$y_k^\ast=\nabla q(y_k)^T\lambda^k$. Due to
\cite[Example 9.44]{RoWe98} we have $\norm{\lambda^k}\leq \kappa
\norm{y_k^\ast}$. Hence the sequence $(\lambda^k)$ is uniformly
bounded. By passing to subsequences if necessary we can assume
that the sequence $(\lambda^k)$ converges to some $\tilde\lambda$
and that there are index sets $\tilde I^+\subset\tilde \I$  such that
$\tilde \I=\I(y_k)$, $\tilde I^+=I^+(\lambda^k)$ $\forall k$.
By virtue of Lemma \ref{LemAux1} we can find some $\tilde\mu$ such that
$(\tilde\lambda,\tilde\mu)\in\Mb(v,v^\ast)$ and $I^+(\tilde\lambda,\tilde\mu)\subset\tilde I^+$.

Taking into account that
\[\nabla q_i(\yb)\tilde v=\lim_{k\to\infty}\frac{q_i(y_k)-q_i(\yb)}{\tilde t_k}\begin{cases}
=0&\mbox{if $i\in\tilde \I$,}\\\leq0&\mbox{if $i\in\Ib\setminus\tilde\I$,}
\end{cases}\]
we obtain $\tilde v\in\Tlin(\yb)$. This, together with $\tilde\lambda\in\Lb$ and $I^+(\tilde\lambda)\subset\tilde I^+\subset \tilde \I$, implies that $\yba^T\tilde v=\tilde\lambda^T\nabla q(\yb)\tilde v=0$ showing $\tilde v\in\Kb$. Further, for each $\lambda\in\Lb$ and every $k$ we have $\lambda^Tq(y_k)\leq 0=\tilde\lambda^Tq(y_k)$ and together with $\lambda^Tq(\yb)=\tilde\lambda^Tq(\yb)=0$ and $\lambda^T\nabla q(\yb)=\tilde\lambda^T\nabla q(\yb)=\yba^T$ we conclude
\begin{eqnarray*}0&\leq& \lim_{k\to\infty}\frac {(\lambda-\tilde\lambda)^Tq(y_k)}{\tilde t_k^2}\\
&=&\lim_{k\to\infty}\frac {(\lambda-\tilde\lambda)^Tq(\yb)+\tilde t_k(\lambda-\tilde\lambda)^T\nabla q(\yb)\tilde v_k+ \frac {\tilde t_k^2}2 \tilde v_k^T\nabla^2((\lambda-\tilde\lambda)^Tq)(\yb)\tilde v_k+\oo(\tilde t_k^2)}{\tilde t_k^2}\\
&=&\frac 12 \tilde v^T\nabla ^2((\lambda-\tilde\lambda)^Tq)(\yb)\tilde v
\end{eqnarray*}
showing  $\tilde\lambda\in\Lb(\tilde v)$.

By Lemma \ref{LemEstActConstr} there is some $\bar z\in \R^m$ with
\[\nabla q_i(\yb)\bar z+\tilde v^T\nabla q_i(\yb)\tilde v\begin{cases}=0&\mbox{if $i\in\tilde \I$}\\\leq0&\mbox{if $i\in \Ib(\tilde v)\setminus\tilde \I$.}\end{cases}\]
By adding some multiple of $\tilde v$ to $\bar z$ we can also assume that $\nabla q_i(\yb)\bar z+\tilde v^T\nabla q_i(\yb)\tilde v\leq 0$ holds for all $i\in\Ib\setminus \Ib(\tilde v)$.
Using the inclusions $I^+(\tilde\lambda)\subset\tilde I^+\subset \tilde \I$ again we obtain
\[0=\sum_{i=1}^l\tilde\lambda_i(\nabla q_i(\yb)\bar z+\tilde v^T\nabla q_i(\yb)\tilde v)=\yba{}^T\bar z+ \tilde v^T\nabla^2(\tilde\lambda^Tq)(\yb)\tilde v\]
showing $\bar z\in \Zb(\tilde v)$. Defining $\J:=\{i\in \Ib(\tilde v)\mv
\nabla q_i(\yb)\bar z+\tilde v^T\nabla q_i(\yb)\tilde v=0\}$, we obtain
$I^+(\tilde\lambda)\subset\tilde I^+\subset\tilde \I\subset \J\in
\Jb(\tilde v)$. By our assumption of 2-LICQ in direction $\tilde v$ the mapping $y\to
(q_i(y))_{i\in\tilde \I}$ is 2--regular in direction $\tilde v$ and
therefore the gradients $\nabla q_i(y_k)$, $i\in\tilde \I$, are
linearly independent by Proposition \ref{Prop2Reg}(d). Hence by Theorem \ref{ThNormalConeLICQ} we have
\[w_k\in K(y_k,y_k^\ast)=\{w\mv  \nabla q_i(y_k)w=0,\ i\in \tilde I^+, \nabla q_i(y_k)w\leq 0, i\in\tilde \I\setminus \tilde I^+\},\]
\[w_k^\ast+ \nabla ^2({\lambda^k}^Tq)(y_k)w_k\in (K(y_k,y_k^\ast))^\circ\]
and it follows  that $\bar w\in \KbIt$. Now consider
for every $s=(s_i)_{i\in \tilde \I}$ and $z^\ast,\tilde
z^\ast\in\R^m$ the set
\begin{eqnarray*}\lefteqn{\Delta(s,z^\ast,\tilde z^\ast)}\\
&:=&\left\{(z,\mu,\nu)\in\R^m\times \PIt\times \PIt \mv \begin{array}{l}\nabla q_i(\yb)z+\tilde v^T\nabla^2 q_i(\yb)\wb\begin{cases}= s_i,& i\in \tilde I^+,\\
\leq s_i,& i\in \tilde \I\setminus \tilde I^+,\end{cases}\\
\nabla q(\yb)^T\mu+\nabla^2(\nu^Tq)(\yb)\tilde v=z^\ast,\\
\nabla q(\yb)^T\nu =\tilde z^\ast
\end{array}\right\}.\end{eqnarray*}
Defining $s^k_i:=(-\nabla q_i(y_k)w_k+ \nabla q_i(\yb)w_k+\tilde t_k
\tilde v^T\nabla^2q_i(\yb)\wb)/\tilde t_k$, $i\in\tilde \I$, we have
$\lim_{k\to\infty}s^k_i=0$ because of $(\nabla q_i(y_k)- \nabla
q_i(\yb)-\tilde t_k \tilde v^T\nabla^2q_i(\yb))/\tilde t_k\to 0$ and $w_k\to\wb$. Since
$(K(y_k,y_k^\ast))^\circ=\{\nabla q(y_k)^T\mu\mv \mu\in \PIt\}$, we can find for each $k$ some vector $\mu^k\in
\PIt$  such that
\[\nabla q(y_k)^T\mu^k=w_k^\ast+\nabla ^2({\lambda^k}^Tq)(y_k)w_k=:z_k^\ast.\]
It follows that the sequence $\nabla q(y_k)^T\mu^k$ is uniformly
bounded by some constant $c$ and by Proposition \ref{Prop2Reg}(d)
we obtain that there is some constant $\kappa'$ such that
$\norm{\mu^k}\leq \kappa'c/\tilde t_k$ $\forall k$. Setting
$r_k^\ast:=(\nabla q(y_k)-\nabla q(\yb) -\tilde t_k\tilde v^T\nabla^2
q(\yb))^T\mu^k$, we have
\[\lim_{k\to\infty} r_k^\ast=\left((\nabla q(y_k)-\nabla q(\yb) -\tilde t_k\tilde v^T\nabla^2 q(\yb))^T/\tilde t_k\right)(\tilde t_k\mu^k)=0,\]
because $(\nabla q(y_k)-\nabla q(\yb) -\tilde t_k\tilde v^T\nabla^2 q(\yb))/\tilde t_k\to 0$ and $\tilde t_k\mu^k$ is bounded.
Defining $\tilde z_k^\ast:=\tilde t_k\nabla g(\yb)^T\mu^k$, we have
\[\lim_{k\to\infty} \tilde z_k^\ast = \lim_{k\to\infty} \tilde z_k^\ast-\tilde t_k z_k^\ast=\lim_{k\to\infty}(\nabla q(\yb)-\nabla q(y_k))^T(t_k\mu^k)=0.\]
Taking into account that $w_k\in K(y_k,y_k^\ast)$  we have
$(w_k/\tilde t_k,\mu^k,\tilde t_k\mu^k)\in\Delta(s^k,z_k^\ast-r_k^\ast,\tilde
z_k^\ast)$ and therefore, by invoking Hoffman's lemma, for every $k$ there is some
$(z_k,\tilde\mu^k,\tilde
\nu^k)\in\Delta(s^k,z_k^\ast-r_k^\ast,\tilde z_k^\ast)$ satisfying
\[\norm{(z_k,\tilde\mu^k,\tilde \nu^k)}=\dist{0,\Delta(s^k,z_k^\ast-r_k^\ast,\tilde z_k^\ast)}\leq\beta(\sum_{i\in\tilde\I}\vert s_i^k-\tilde v^T\nabla^2q_i(\yb)\wb\vert +\norm{z_k^\ast-r_k^\ast}+\norm{\tilde z_k^\ast})\]
with some constant $\beta$ independent of $k$.
Since the sequences $(s^k)$, $(z_k^\ast-r_k^\ast)$ and $(\tilde
z_k^\ast)$ are bounded, so also is the sequence
$(z_k,\tilde\mu^k,\tilde \nu^k)$ and, by passing to a subsequence,
it converges to some $(\hat z,\hat\mu,\hat\nu)$. Since
$\lim_k s^k=0$, $\lim_k z_k^\ast-r_k^\ast=\wb^\ast+\nabla
^2({\tilde\lambda}^Tq)(\yb)\wb$ and $\lim_k \tilde z_k^\ast=0$, we
have $(\hat z,\hat\mu,\hat\nu)\in\Delta(0,\wb^\ast+\nabla
^2({\tilde\lambda}^Tq)(\yb)\wb,0)$ showing the desired inclusions
$\wb\in \KbIt(\tilde v)$ and $\wb^\ast+\nabla
^2({\tilde\lambda}^Tq)(\yb)\wb\in (\KbIt(\tilde v))^\circ$. This completes the first part of the proof.

In the second part of the proof we show equality in the inclusions \eqref{EqInclLimNormalCone1}, \eqref{EqInclLimNormalCone2} under the stated assumptions. In case $v=0$ we choose
any $\tilde v$ from $\Kb$ with $\norm{\tilde v}=1$, otherwise we set $\tilde v:=v$.
Then we consider  multipliers $(\tilde\lambda,\tilde\mu)\in \Mb(v,v^\ast)$ with $\tilde\lambda\in\Lb(\tilde v)$,
index sets  $\tilde I^+$, $\tilde \I$, $\J$ with
$I^+(\tilde\lambda,\tilde\mu)\subset\tilde I^+\subset\tilde \I\subset\J\in\Jb(\tilde v)$ and elements $\wb$, $\wb^\ast$ with $\wb\in \KbIt(\tilde v)$,
$\wb^\ast+\nabla^2(\tilde\lambda^Tq)(\yb)\wb\in (\KbIt(\tilde v))^\circ$. We will show that for every $t>0$ sufficiently
small there are $(y_t,y_t^\ast,w_t,w_t^\ast)$ with $y_t\not=\yb$,
$(w_t^\ast,w_t)\in \widehat N_{\Gr \widehat N_\Gamma}(y_t,y_t^\ast)$ such
that $\lim_{t\downarrow
0}(y_t,y_t^\ast,w_t,w_t^\ast)=(\yb,\yba,\wb,\wb^\ast)$, $\lim_{t\downarrow 0}((y_t,y_t^\ast)-(\yb,\yba))/t=(v,v^\ast)$ and hence the claimed inclusion
$(\wb^\ast,\wb)\in N_{\Gr \widehat N_\Gamma}((\yb,\yba);(v,v^\ast))$ follows. We can assume without
loss of generality that $\J$ is a maximal element in $\Jb(\tilde v)$ with $\tilde \I\subset\J$. Then, by Proposition \ref{Prop2LICQ} there exists some $\bar\tau>0$ and a mapping $\hat y:[0,\bar\tau]\to \Gamma$ such that $\hat y(0)=\yb$, $\I(\hat y(\tau))=\tilde \I$, LICQ is fulfilled at $\hat y(\tau)$ for every $\tau\in (0,\bar\tau)$  and
\[\lim_{\tau\downarrow 0}\frac {\hat y(\tau)-\yb}\tau=\tilde v.\]
We now define
\[\tau(t):=t\frac{t^2+\norm{v}}{t+\norm{v}},\ y_t:=\hat y(\tau(t))\]
and observe that $\lim_{t\downarrow 0}(y_t-\yb)/t=v$.
Next we define the multipliers $\lambda^t$ by
\[\lambda^t_i:=\begin{cases}\tilde\lambda_i+t\tilde\mu_i+t^2,& i\in\tilde I^+,\\0,&i\not\in\tilde I^+\end{cases}\]
and then it follows from $\tilde\mu\in T_{N_{\R^l_-}(q(\yb))}(\tilde\lambda)$ that $\lambda^t\geq 0$ for all $t>0$ sufficiently small.
Defining $y_t^\ast:=\nabla q(y_t)^T\lambda^t$ we obtain
\[\lim_{t\downarrow 0}\frac{y_t^\ast-\yba}t=
\lim_{t\downarrow 0}\frac{(\nabla q(y_t)-\nabla q(\yb))^T\tilde\lambda}t+\nabla q(y_t)^T\tilde\mu=\nabla^2(\tilde\lambda^Tq)(\yb)v+\nabla q(\yb)^T\tilde\mu=v^\ast,\]
and, since $\I(y_t)=\tilde\I$ and $I^+(\lambda^t)=\tilde I^+$, we have
$K(y_t,y_t^\ast)=\{w\mv \nabla q_i(y_t)w=0,\ i\in\tilde I^+, q_i(y_t)w\leq0,\ i\in\tilde\I\setminus\tilde I^+\}$,
and $(K(y_t,y_t^\ast))^\circ=\nabla q(y_t)^TP_{\tilde I^+,\tilde \I}$. Let $z$ be some element associated with $\wb$ by the
definition of  $K_{\tilde I^+,\tilde \I}(\tilde v)$. Then
\begin{eqnarray*}\lefteqn{\nabla q_i(\hat y(\tau))(\wb +\tau z)}\\
&=&(\nabla q_i(\yb)+\tau\tilde v^T\nabla^2q_i(\yb))^T(\wb+\tau z)+ \oo(\tau)\\
&=&\nabla q_i(\yb)\wb+\tau(\nabla q_i(\yb)z+\tilde v^T\nabla^2
q_i(\yb)\wb)+\oo(\tau)
\begin{cases}=\oo(\tau),& i\in\tilde I^+,\\\leq\oo(\tau),&i\in\tilde \I\setminus \tilde I^+,\end{cases}
\end{eqnarray*}
implying $\norm{s(\tau)}=\oo(\tau)$ where $s_i(\tau):=-\nabla q_i(\hat y(\tau))(\wb
+\tau z)$ for $i\in\tilde I^+$ and $s_i(\tau):=-\max\{\nabla q_i(\hat y(\tau))(\wb
+\tau z),0\}$ for $i\in \tilde\I\setminus\tilde I^+$. Using 2--regularity of
$(q_i)_{i\in\tilde \I}$ in direction $\tilde v$, by means of Proposition
\ref{Prop2Reg}(d)
\if{and using the fact that for every matrix $A$ and
every right hand side $b$ we have
\[\min\{\norm{x}\mv Ax=b\}\leq \sup\{\frac 1{\norm{A^T\lambda}}\mv \norm{\lambda}=1\}\norm{b},\]}\fi
we can find for all $\tau>0$ sufficiently small some $e_\tau$ with
$\nabla q_i(\hat y(\tau))e_\tau=s(\tau)$ and
\[\lim_{\tau\downarrow 0}\norm{e_\tau}\leq\lim_{t\downarrow 0}\frac{\kappa'}{\norm{\hat y(\tau)-\yb}}\norm{s(\tau)}=0,\]
implying $w_t:=\wb+\tau(t)z+e_{\tau(t)}\in K(y_t,y_t^\ast)$ and
$\lim_{t\downarrow 0}w_t=\wb$.

Finally we choose $\bar\mu,\bar\nu\in P_{\bar I^+,\bar I}$ such
that $\nabla q(\yb)^T\bar\nu=0$ and $\wb^\ast
+\nabla^2(\tilde\lambda^Tq)(\yb)\wb=\nabla
q(\yb)^T\bar\mu+\nabla^2(\bar\nu^Tq)(\yb)\tilde v$. Taking
$\mu_\tau:=\bar\mu+\bar\nu/\tau$ we have $\mu_\tau\in P_{\tilde I^+,\tilde \I}$ and
\begin{eqnarray*}\lim_{\tau\downarrow 0}\nabla q(\hat y(\tau))^T\mu_\tau&=&\lim_{\tau\downarrow 0}(\nabla q(\yb)+\tau \tilde v^T\nabla ^2 q(\yb)+\oo(\tau))^T\mu_\tau\\
&=&\lim_{\tau\downarrow 0}\tau^{-1}\nabla q(\yb)^T\bar\nu+ \nabla
q(\yb)^T\bar\mu+\nabla^2(\bar\nu^Tq)(\yb)\tilde v \\
&=&\nabla
q(\yb)^T\bar\mu+\nabla^2(\bar\nu^Tq)(\yb)\tilde v.
\end{eqnarray*}
Defining $w_t^\ast=\nabla
q(\hat y(\tau(t)))^T\mu_{\tau(t)}-\nabla^2({\lambda^t}^Tq)(y_t)w_t$ we have
$\lim_{t\downarrow 0}w_t^\ast=\wb^\ast$ and, because of $\nabla q(\hat y(\tau(t)))^T\mu_{\tau(t)} \in (K(y_t,y_t^\ast))^\circ$, one has $(w_t^\ast,w_t)\in
\widehat N_{\Gr \widehat N_\Gamma}(y_t,y_t^\ast)$. This completes the proof.
\end{proof}
To compute a suitable estimate of $N_{\Gr\widehat N_\Gamma}((\yb,\yba);(0,v^\ast))$, we turn now our attention to the cone $N_{\Gr\widehat N_\Gamma}^2((\yb,\yba);(0,v^\ast))$.
\begin{proposition}\label{PropRegLimCone}Let  $v^\ast\not=0$ such that $(0,v^\ast)\in T_{\Gr\widehat N_\Gamma}(\yb,\yba)$ and
assume that $M_q$ is metrically subregular at $\yb$. If $\Kb\not=\{0\}$, then
\begin{eqnarray}\nonumber\lefteqn{N_{\Gr\widehat N_\Gamma}^2((\yb,\yba);(0,v^\ast))}\\
&\subset& \bigcap_{\tilde v\in\Null(\yb)}\left(
\bigcup\limits_{\AT{(\lambda,\mu)\in\Mb(0,v^\ast):\lambda\in\bar{\tilde\Lambda}^\E(\tilde v)}{\J\in\Jb(\tilde v)}}
\bigcup\limits_{I^+(\lambda,\mu)\subset I^+\subset \I\subset \J}Q_0(\tilde v,\lambda,I^+,\I)\right),\label{EqInclLimRegNormalCone}\end{eqnarray}
where
\[Q_0(v,\lambda,I^+,\I):=\{(w^\ast,w)\mv w\in \KbI(v),\  w^\ast+\nabla^2(\lambda^Tq)(\yb)w\in \nabla q(\yb)^TP_{I^+,\I}\}.\]
Further, if $\Kb=\{0\}$, then
\begin{equation}N_{\Gr\widehat N_\Gamma}^2((\yb,\yba);(0,v^\ast))\subset\R^m\times\{0\}.\label{EqInclLimRegNormalConeKritConeEmpty}\end{equation}
\end{proposition}

\begin{proof}
Let $(\wb^\ast,\wb)\in N_{\Gr\widehat N_\Gamma}^2((\yb,\yba);(0,v^\ast))$ and consider sequences $(w_k^\ast, w_k)\to (\wb^\ast,\wb)$, $(t_k)\downarrow 0$,
$(v_k^\ast)\to v^\ast$ such that $(w_k^\ast, w_k)\in\widehat N_{\Gr\widehat N_\Gamma}(\yb,y_k^\ast)$
where $y_k^\ast:=\yba+t_k v_k^\ast$. Consider first the case when $\Kb\not=\{0\}$ and $K(\yb,y_k^\ast)\not =\{0\}$ for infinitely many $k$ and let $\tilde v\in\Null(\yb)$ be fixed.
We will now show that there are
multipliers  $(\tilde \lambda,\tilde\mu)\in\Mb(0,v^\ast)$ with $\tilde\lambda\in\Lb^\E(\tilde v)$ and index sets $\tilde  I^+$,$\tilde \I$, $\J$ with
 $I^+(\tilde\lambda,\tilde\mu)\subset\tilde I^+\subset\tilde \I\subset\J\in\Jb(\tilde v)$ such  that $\wb\in \KbIt(\tilde v)$ and  $\wb^\ast+\nabla^2(\tilde\lambda^Tq)(\yb)\wb\in \nabla q(\yb)^T\PIt$.

By passing to a subsequence we can assume that $K(\yb,y_k^\ast)\not =\{0\}$ holds for all $k$.  By Theorem \ref{ThRegNormalCone}  we have
\[w_k\in {\cal W}(\yb,y_k^\ast;\tilde v):=\{w\in K(\yb,y_k^\ast)\mv w^T\nabla^2((\lambda^1-\lambda^2)^Tq)(\yb)\tilde v=0\ \forall
\lambda^1,\lambda^2\in\Lambda(\yb,y_k^\ast;\tilde v)\}\]
and there is some
\[\lambda^k\in\tilde\Lambda^\E(\yb,y_k^\ast;\tilde v):=\begin{cases}\Lambda^\E(\yb,y_k^\ast;\tilde v)
&\mbox{if $\tilde v\not=0$,}\\
\co(\bigcup\limits_{0\not=u\in K(\yb,y_k^\ast)}\Lambda^\E(\yb,y_k^\ast;u))&\mbox{if $\tilde v=0$,}\end{cases}\]
such that
\[w_k^\ast\in -\nabla^2({\lambda^k}^Tq)(\yb)w_k+(K(\yb,y_k^\ast))^\circ.\]
By \eqref{EqBndExtremePoints} there is some $\kappa>0$  such that $\E(\yb,y_k^\ast)$
 is contained in a ball with radius $\kappa\norm{y_k^\ast}$. Hence the sequence $(\lambda^k)$ is uniformly
bounded. By passing to subsequences if necessary we can assume
that the sequence $(\lambda^k)$ converges to some $\tilde\lambda$
and that there is some index set $\tilde I^+$  such that
$\tilde I^+=I^+(\lambda^k)$ $\forall k$.
By Lemma \ref{LemAux1}  we can find some $\tilde\mu$ such that
$(\tilde\lambda,\tilde\mu)\in\Mb(0,v^\ast)$ and $I^+(\tilde\lambda,\tilde\mu)\subset\tilde I^+$.
Since
\[K(\yb,y_k^\ast)=\left\{v\in \R^m\mv \begin{array}{l}
\nabla q_i(\yb)v=0,\ i\in  I^+(\lambda^k)\\
\nabla q_i(\yb)v\leq 0,\ i\in \Ib\setminus I^+(\lambda^k)
\end{array}\right\},\]
for every $k$, there is some $\mu\in P_{\tilde I^+,\Ib}$ with $w_k^\ast+\nabla^2({\lambda^k}^Tq)(\yb)w_k=\nabla q(\yb)^T\mu$.
Now consider the linear optimization problem
\begin{equation}\label{EqLinProblMu}
\min_\mu -\tilde v^T\nabla^2(\mu^Tq)(\yb)\tilde v\quad\mbox{ subject to} \quad w_k^\ast+\nabla^2({\lambda^k}^Tq)(\yb)w_k=\nabla q(\yb)^T\mu, \mu\in P_{\tilde I^+,\Ib}.
\end{equation}
This problem has some solution, since the feasible region is not empty and the objective is bounded below on the feasible region. Indeed, otherwise there would be
some $\nu\in P_{\tilde I^+,\Ib}$ such that $\nabla q(\yb)^T\nu=0$, $\tilde v^T\nabla^2(\nu^Tq)(\yb)\tilde v>0$ and consequently
$\lambda^k+\alpha\nu\in\Lambda(\yb,y_k^\ast)$ and $\tilde v^T\nabla^2((\lambda^k+\alpha\nu)^Tq)(\yb)\tilde v>\tilde v^T\nabla^2({\lambda^k}^Tq)(\yb)\tilde v$
for $\alpha>0$ sufficiently small contradicting $\lambda^k\in\Lambda^\E(\yb,y_k^\ast;\tilde v)$.
By duality theory
of linear programming, the dual problem
\[\max_z (w_k^\ast+\nabla^2({\lambda^k}^Tq)(\yb)w_k)^Tz\quad\mbox{ subject to }\nabla q_i(\yb)z+\tilde v^T\nabla^2 q_i(\yb)\tilde v
\begin{cases}=0,&i\in\tilde I^+,\\\leq 0,&i\in\Ib\setminus \tilde I^+\end{cases}\]
also has a solution $z_k$ which, together with any solution $\mu$ of \eqref{EqLinProblMu} fulfills  the complementarity condition $\mu_i(\nabla q_i(\yb)z_k+\tilde v^T\nabla^2 q_i(\yb)\tilde v)=0$, $i\in\Ib$.
We now select $\mu^k$ among the solutions of the problem \eqref{EqLinProblMu} such that
the cardinality of the index set $J^+(\mu^k):=\{i\in \Ib\setminus \tilde I^+:\mu^k_i>0\}$ is minimal.
Then
\[\nabla q(\yb)^T\nu=0, \nu\in P_{\tilde I^+, \tilde I^+\cup J^+(\mu_k)} \ \Rightarrow\ \nu_i=0,\ i\in J^+(\mu_k),\]
because otherwise we can find some scalar $\alpha$ such that $\mu^k+\alpha\nu\in P_{\tilde I^+, \tilde I^+\cup J^+(\mu_k)}\subset P_{\tilde I^+,\Ib}$
is feasible for \eqref{EqLinProblMu} and
$J^+(\mu_k+\alpha\nu)\subset  J^+(\mu_k)$. This  shows that $\mu^k+\alpha\nu$ is a solution of \eqref{EqLinProblMu}
because the complementarity condition remains fulfilled,  and $\vert J^+(\mu_k+\alpha\nu)\vert < \vert J^+(\mu_k)\vert$,
contradicting the minimality of $\vert J^+(\mu^k)\vert$.

By eventually passing to a subsequence once more, we can assume that
$J^+(\mu_k)=J^+$ holds for all $k$ and we set $\tilde \I:=\tilde I^+\cup J^+$. Fixing $z=z_1$, we obtain from the complementarity condition that
\begin{equation}\label{EqInequ1}\nabla q_i(\yb)z+\tilde v^T\nabla^2 q_i(\yb)\tilde v\begin{cases}=0&\mbox{if $i\in\tilde \I$}\\\leq 0&\mbox{if $i\in\Ib\setminus \tilde \I$,}\end{cases}\end{equation}
and therefore
\[\tilde \lambda_i(\nabla q_i(\yb)z+\tilde v^T\nabla^2 q_i(\yb)\tilde v)=
\lim_{k\to\infty}\lambda^k_i(\nabla q_i(\yb)z+\tilde v^T\nabla^2 q_i(\yb)\tilde v)=0.\]
Hence the pair $(\tilde\lambda, z)$ is feasible for \eqref{EqLPDirMult} and its dual \eqref{EqDPDirMult} at $(\yb,\yba)$ and fulfills the complementarity condition, implying by duality theory of linear programming that $\tilde\lambda\in\Lb(\tilde v)$ and $z\in \Zb(\tilde v)$.
Since $\tilde v\in\Null(\yb)$, we have  $\Ib(\tilde v)=\Ib$ and  we put $\J:=\{i\in\Ib\mv \nabla q_i(\yb)z+\tilde v^T\nabla^2 q_i(\yb)\tilde v)=0\}$.

In a next step we show that $\tilde\lambda\in\bar{\tilde\Lambda}^\E(\tilde v)$.
The multiplier $\lambda^k$ is the convex combination of finitely many extreme points $\hat\lambda^{k,j}\in \Lambda(\yb,y_k^\ast;u^k_j)\cap\E(\yb,y_k^\ast)$, $j=1,\ldots,p_k$, where $0\not=u^k_j\in K(\yb;y_k^\ast)$, where $u^k_j=\tilde v$ $\forall k,j$ if $\tilde v\not=0$,  and, since $\Lambda(y,y^\ast;\alpha u)=\Lambda(y,y^\ast;u)$ $\forall\alpha>0$, we can assume that $\norm{u^k_j}=1$ in case $\tilde v=0$. By passing to subsequences we can also assume that $p_k=\bar p$ $\forall k$ and $u^k_j\to\bar u_j$, $j=1,\ldots,\bar p$, as $k\to\infty$ and $I^+(\hat\lambda^{k,j})=I^+_j$, $j=1,\ldots,\bar p$, holds for all $k$. It follows that for each $j$ the sequence $\hat \lambda^{k,j}$ converges to some $\hat\lambda^j\in\Lb$ with $I^+(\hat\lambda^j)\subset I^+_j$ and thus $\hat\lambda^j$ is an extreme point of $\Lb$. Hence $\tilde \lambda$ is a convex combination of these $\hat\lambda^j$, $j=1,\ldots,\bar p$,  $\bar u_j\in \Kb$ and since $\hat\lambda^j\in\Lb(\bar u_j)$, $j=1,\ldots,\bar p$, because of \cite[Theorem 5.4.2(2)]{BaGuKlKuTa82}, we obtain $\hat\lambda^j\in\Lb^\E(\bar u_j)$ and thus $\tilde\lambda\in\bar{\tilde\Lambda}^\E(0)$. In case that $\tilde v\not=0$ we have $u^k_j=\tilde v$ $\forall j,k$  and $\tilde\lambda\in\bar{\tilde\Lambda}^\E(\tilde v)$ follows.

It remains to show that $\bar w\in\KbIt(\tilde v)$ and $\wb^\ast+\nabla^2(\tilde\lambda^Tq)(\yb)\wb\in \nabla q(\yb)^T\PIt$. Let us first prove  by contradiction that $\bar w\in\KbIt(\tilde v)$. Assuming that $\bar w\not\in\KbIt(\tilde v)$, by the Farkas Lemma there is some $\nu\in\PIt$ with $\nabla q(\yb)^T\nu=0$ and $\tilde v^T\nabla^2(\nu^Tq)(\yb)\bar w>0$, yielding $\tilde v^T\nabla^2(\nu^Tq)(\yb)w_k>0$ for all $k$ sufficiently large. From \eqref{EqInequ1} we deduce $\tilde v^T\nabla^2(\nu^Tq)(\yb)\tilde v=0$. Hence, for every $k$ sufficiently large there is $\alpha_k>0$ such that $\lambda^k+\alpha_k\nu\in\Lambda(\yb,y_k^\ast;\tilde v)$ and $w_k^T\nabla^2(((\lambda^k+\alpha_k\nu)-\lambda^k)^Tq)(\yb)\tilde v>0$, contradicting $w_k\in{\cal W}(\yb,y_k^\ast;\tilde v)$. Hence, the desired inclusion
$\bar w\in\KbIt(\tilde v)$ holds true. Finally note that, by the way we constructed the index set $\tilde \I$, for every $k$ there is some $\mu^k\in\PIt$ satisfying $w_k^\ast+\nabla^2({\lambda^k}^Tq)(\yb)w_k=\nabla q(\yb)^T\mu^k$. Utilizing Hoffman's Error Bound there is some constant $\beta$ such that for every $k$ there is also an element $\tilde\mu^k\in\PIt$ such that $w_k^\ast+\nabla^2({\lambda^k}^Tq)(\yb)w_k=\nabla q(\yb)^T\tilde \mu^k$ and $\norm{\tilde\mu^k}\leq\beta\norm{w_k^\ast+\nabla^2({\lambda^k}^Tq)(\yb)w_k}$. Thus the sequence $(\tilde\mu^k)$ is bounded and we can assume that it converges to some $\tilde\mu\in \PIt$ satisfying $\wb^\ast+\nabla^2(\tilde\lambda^Tq)(\yb)\wb=\nabla q(\yb)^T\tilde\mu$. This completes the proof of the case when $K(\yb,y_k^\ast)\not=\{0\}$ for all $k$.

In a next step we consider the case that $\Kb\not=\{0\}$ and $K(\yb,y_k^\ast)\not=\{0\}$ only holds for finitely many $k$. Without loss of generality we can assume that we have $K(\yb,y_k^\ast)=\{0\}$ and consequently $w_k=0$ $\forall k$. We observe that we always have $\Null(\yb)\subset K(\yb,y_k^\ast)$ and thus $\Null(\yb)=\{0\}$ and we will proceed as in the first part of the proof with the only difference in the choice of the sequence $(\lambda^k)$. Pick an arbitrary $0\not=u\in \Kb$. Then, since $M_q$ is metrically subregular at $(\yb,0)$, by \cite[Theorem 6.1(2b)]{Gfr11} for every $\lambda\in N_{R^l_-}(q(\yb)$ with $\nabla q(\yb)^T\lambda=0$ we have $u^T\nabla^2 (\lambda^Tq)(\yb)u\leq 0$. We obtain that the linear program
\begin{equation}\label{EqLinProblAux}
\max u^T\nabla^2 (\lambda^Tq)(\yb)u\mbox{ subject to }\lambda\in\Lambda(\yb,y_k^\ast)
\end{equation}
has a solution and we select $\lambda^k\in\Lambda(\yb,y_k^\ast;u)\cap\E(\yb,y_k^\ast)$. This can be done since among the solutions of a linear optimization problem there is always an extreme point, provided the feasible region has at least one extreme point. Then the same arguments as before yield the assertion.

Finally, let us consider the case $\Kb=\{0\}$. Given an arbitrary element $(w^\ast,w)\in N_{\Gr\widehat N_\Gamma}^2((\yb,\yba);(0,v^\ast))$, we consider  sequences $(t_k)\downarrow 0$, $v_k^\ast\to v^\ast$ and $(w_k^\ast,w_k)\to (w^\ast,w)$  such that $(w_k^\ast,w_k)\in\widehat N_{\Gr\widehat N_\Gamma}(\yb,y_k^\ast)$, where $y_k^\ast:=y^\ast+t_k v_k^\ast$. We will now show by contraposition that  $K(\yb,y^\ast_k)=\{0\}$ holds for all $k$ sufficiently large. Assume on the contrary that for every $k$ there is some $z_k\in K(\yb,y_k^\ast)$ with $\norm{z_k}=1$. Then, by passing to a subsequence we can assume that $(z_k)$ converges to some $z$. Because $z_k\in \Tlin(\yb)$ and $\Tlin(\yb)$ is closed, we have $z\in \Tlin(\yb)$ and, since $\yba^Tz=\lim {y_k^\ast}^Tz_k=0$, it follows that $0\not=z\in \Kb$, a contradiction.  Hence, $K(\yb,y^\ast_k)=\{0\}$ and from \eqref{EqInclNormalCone} we conclude $\widehat N_{\Gr\widehat N_\Gamma}(\yb,y_k^\ast)\subset \R^m\times\{0\}$. It follows that $w_k=0$ implying $w=0$ and this completes the proof.
\end{proof}

We do not give a characterization when  equality holds in \eqref{EqInclLimRegNormalCone} and \eqref{EqInclLimRegNormalConeKritConeEmpty}, respectively, because in many cases we have $N_{\Gr\widehat N_\Gamma}^2((\yb,\yba);(0,v^\ast))\subset N_{\Gr\widehat N_\Gamma}^1((\yb,\yba);(0,v^\ast))$ and for the latter set an exact description is known. This issue is clarified in the next statement.
\begin{proposition}\label{PropRegContLim}Assume that $M_q$ is metrically subregular at $(\yb,0)$  and  metrically regular in the vicinity of $\yb$. Further assume that for every direction $0\not=u\in \Kb$ and every maximal index set $\J\in\Jb(u)$  the mapping $y\to (q_i(y))_{i\in
\J}$ is 2--regular at $\yb$ in direction $u$ and assume that $\Null(\yb)\not=\{0\}$. Then for every $v^\ast\not=0$ one has
\[N_{\Gr\widehat N_\Gamma}^2((\yb,\yba);(0,v^\ast))\subset N_{\Gr\widehat N_\Gamma}^1((\yb,\yba);(0,v^\ast)).\]
\end{proposition}

\begin{proof}By Lemma \ref{LemDualCone} one has that $\nabla q(\yb)^T\PI\subset(\KbI(v))^\circ$ and, consequently, $Q_0(v,\lambda,I^+,\I)\subset Q(v,\lambda,I^+,\I)$. Now it is easy to see that in case $\Null(\yb)\not=\{0\}$ the set on the right hand side of the inclusion \eqref{EqInclLimRegNormalCone} is a subset of the set on the right hand side of \eqref{EqInclLimNormalCone2}. Hence the inclusion $N_{\Gr\widehat N_\Gamma}^2((\yb,\yba);(0,v^\ast))\subset N_{\Gr\widehat N_\Gamma}^1((\yb,\yba);(0,v^\ast))$ follows from Theorem \ref{ThLimNormalCone1} and Proposition \ref{PropRegLimCone}.
\end{proof}

We summarize these results in the following theorem to give a complete description of the limiting normal cone:

\begin{theorem}\label{ThEqLimNormalCone}
Assume that $M_q$ is metrically subregular at $(\yb,0)$  and  metrically regular in the vicinity of $\yb$. Further assume that for every direction $0\not=u\in \Kb$ and every maximal index set $\J\in\Jb(u)$ the mapping $y\to (q_i(y))_{i\in
\J}$ is 2--regular at $\yb$ in direction $u$ and assume that $\Null(\yb)\not=\{0\}$. Then
\begin{eqnarray*}
\lefteqn{N_{\Gr \widehat N_\Gamma}(\yb,\yba)}\\
&=&\widehat N_{\Gr \widehat N_\Gamma}(\yb,\yba)\cup\bigcup_{\AT{(v,v^\ast)\in T_{\Gr N_\Gamma}(\yb,\yba)}{v\not=0}}\bigcup\limits_{\AT{(\lambda,\mu)\in\Mb(v,v^\ast)}{\J\in\Jb(v)}}
\bigcup\limits_{I^+(\lambda,\mu)\subset I^+\subset \I\subset\J} Q(v,\lambda,I^+,\I).
\end{eqnarray*}
\end{theorem}
\begin{proof}
The statement follows from Theorem \ref{ThLimNormalCone1} and Proposition \ref{PropRegContLim} together with the observation that for any element $(w^\ast,w)\in N_{\Gr\widehat N_\Gamma}^1((\yb,\yba);(0,v^\ast))$ there is some $\tilde v\in \Kb$ with $\norm{\tilde v}=1$, $(\lambda,\mu)\in\Mb(0,v^\ast)$ with $\lambda\in\Lb(\tilde v)$ and index sets $J\in\Jb(\tilde v)$, $I^+$ and $\I$ with $I^+(\lambda,\mu)\subset I^+\subset \I\subset\J$ such that $(w^\ast,w)\in Q(\tilde v,\lambda,I^+,\I)$. Consequently, $(\lambda,\mu)\in\Mb(\tilde v,v^\ast+\nabla^2(\lambda^Tq)(\yb)\tilde v)$, showing $(w^\ast,w)\in N_{\Gr\widehat N_\Gamma}((\yb,\yba);(\tilde v,v^\ast+\nabla^2(\lambda^Tq)(\yb)\tilde v))$.
\end{proof}

We conclude this section with two illustrative examples, the results of which will then be used in the next section.

\begin{example}\label{ExampleIncl}
Let $\Gamma\subset\R^2$ be given by
$$
q(y)=
\begin{pmatrix}
-y_1^2+y_2\\
-y_1^2-y_2\\
y_1
\end{pmatrix}.
$$
Put $\bar{y}=(0,0), \yba=(0,1)$ and let us compute
$N_{\Gr\widehat{N}_\Gamma}(\yb,\yba)$. Obviously, MFCQ is violated at $\yb$. Owing to \cite[Example 4]{GfrOut14} we have $\Kb=\R_-\times\{0\}$,
\[
\Lb=\{\lambda \in \mathbb{R}^{3}_{+}\mv \lambda_{1}- \lambda_{2}=1, \lambda_{3}=0\}
\]
%%%
and
\[
\Lb(v)= \left\{
\begin{array}{lll}
\{\lambda \in \Lb \mv \lambda_{1}=1, \lambda_{2}=0\} & \mbox{ if } & 0 \neq v \in \Kb \\
\bar{\Lambda} & \mbox{ if } & v=0.
\end{array}\right.
\]
Further, $M_q$ is metrically subregular at $(0,0)$ and  metrically regular in the vicinity of $0$
and by Theorem \ref{ThRegNormalCone} we obtain
\[T_{\Gr\widehat{N}_\Gamma}(\yb,\yba )=\{(v,v^\ast)\mv v_1\leq 0, \, v_2 = 0,\, v_1^\ast=-2v_1\}\cup (\{0,0\}\times \R_+\times \R)\]
and
\[\widehat N_{\Gr \widehat N_\Gamma}(\yb,\yba)=\{(w^\ast,w)\mv w_1\leq 0, \, w_2 = 0,\, w_1^\ast\geq 2w_1\}.\]
Now consider $0\not=v\in\Kb$. It follows that $v=(v_1,0)$ with $v_1<0$, $\Ib(v)=\{1,2\}$ and that $\Jb(v)$ consists of the collection of all index sets $\J\subset\{1,2\}$ such that there exists $z$ with
\begin{eqnarray}
\label{EqEx3Eq1}\nabla q_1(\yb)z+v^T\nabla^2q_1(\yb)v= z_2-2v_1^2&\leq& 0\\
\label{EqEx3Eq2}\nabla q_2(\yb)z+v^T\nabla^2q_2(\yb)v= -z_2-2v_1^2&\leq& 0\\
\nonumber\nabla q_3(\yb)z+v^T\nabla^2q_3(\yb)v=z_1&\leq&0\end{eqnarray}
and $\J$ contains the active inequalities of \eqref{EqEx3Eq1}, \eqref{EqEx3Eq2}. Hence, $\Jb(v)=\{\emptyset,\{1\},\{2\}\}$. Since $\nabla q_i(y)=(-2y_1,\pm1)\not=0$, $i=1,2$, for every $\J\subset\Jb(v)$ the mapping $(q_i)_{i\in \J}$ is 2-regular in direction $v$, implying that 2-LICQ holds in direction $v$ by Proposition \ref{PropSuffCond2LICQ}.\\
Further, for every $(v,v^\ast)\in T_{\Gr\widehat{N}_\Gamma}(\yb,\yba)$ with $v\not=0$ we have $v\in\Kb$, $v_1^\ast=-2v_1$ and thus \begin{eqnarray*}\Mb(v,v^\ast)&=&\{(1,0,0)\}\times\{(\mu_1,\mu_2,\mu_3)\mv \mu_2\geq0,\ \mu_1-\mu_2=v_2^\ast, 0\leq \mu_3=v_1^\ast +2v_1\}\\
&=&\{(1,0,0)\}\times\{(\mu_1,\mu_2,0)\mv \mu_2\geq0,\ \mu_1-\mu_2=v_2^\ast\},\end{eqnarray*}
yielding
\[N_{\Gr \widehat N_\Gamma}((\yb,\yba);(v,v^\ast))=Q(v,(1,0,0),\{1\},\{1\})\]
by Theorem \ref{ThLimNormalCone1}, where we have taken into account that the only index set $\J\in\Jb(v)$ with $I^+(1,0,0)=\{1\}\subset \J$ is $\J=\{1\}$. Straightforward calculations give
\begin{eqnarray*}&\Kb_{\{1\},\{1\}}(v)=\Kb_{\{1\},\{1\}}=\R\times\{0\},&\\
&Q(v,(1,0,0),\{1\},\{1\})=\{(w^\ast,w)\mv w_2 = 0, \, w_1^\ast= 2w_1\}.&\end{eqnarray*}
In the next step we want to analyze $N^1_{\Gr \widehat N_\Gamma}((\yb,\yba);(0,v^\ast))$ for $(0,0)\not=(0,v^\ast)\in
T_{\Gr \widehat N_\Gamma}(\yb,\yba)$. It follows that $v_1^\ast\geq 0$ and  for every $\tilde v\in\Kb$, $\norm{\tilde v}=1$, we obtain
\[\{(\lambda,\mu)\in \Mb(0,v^\ast)\mv \lambda\in\Lb(\tilde v)\}=\{(1,0,0)\}\times\{(\mu_1,\mu_2,v_1^\ast)\mv \mu_2\geq0,\ \mu_1-\mu_2=v_2^\ast\}.\]
Since $\Jb(\tilde v)=\{\emptyset,\{1\},\{2\}\}$, if $v_1^\ast>0$ we obtain $N^1_{\Gr \widehat N_\Gamma}((\yb,\yba);(0,v^\ast))=\emptyset$. On the other hand, if $v_1^\ast=0$, similar arguments as before yield
\[N^1_{\Gr \widehat N_\Gamma}((\yb,\yba);(0,v^\ast))=\{(w^\ast,w)\mv w_2 = 0, \, w_1^\ast= 2w_1\}.\]
Finally we consider $N^2_{\Gr \widehat N_\Gamma}((\yb,\yba);(0,v^\ast))$ for $(0,0)\not=(0,v^\ast)\in T_{\Gr \widehat N_\Gamma}(\yb,\yba)$. We have $\Null(\yb)=\{0\}$, $\bar{\tilde \Lambda}^\E(0)=\{(1,0,0)\}$,
\[\{(\lambda,\mu)\in \Mb(0,v^\ast)\mv \lambda\in\bar{\tilde \Lambda}^\E(0)\}=\{(1,0,0)\}\times\{(\mu_1,\mu_2,v_1^\ast)\mv \mu_2\geq0,\ \mu_1-\mu_2=v_2^\ast\}\]
and $\Jb(0)=\{\{1,2\},\{1,2,3\}\}$.\\
Using Proposition \ref{PropRegLimCone} we obtain
\[N_{\Gr\widehat N_\Gamma}^2((\yb,\yba);(0,v^\ast))\subset
\begin{cases}\bigcup_{\{1\}\subset I^+\subset \I\subset\{1,2,3\}}Q_0(0,(1,0,0),I^+,\I)&\mbox{if $v_1^\ast=0$}\\
\bigcup_{\{1,3\}\subset I^+\subset \I\subset\{1,2,3\}}Q_0(0,(1,0,0),I^+,\I)&\mbox{if $v_1^\ast>0$.}
\end{cases}\]
By the definition we have $Q_0(0,(1,0,0),I^+,\I)=\{(w^\ast,w)\mv w\in\KbI, w^\ast-(2w_1,0)\in \KbI^\circ\}$ and
\[\KbI=\begin{cases}(0,0)&\mbox{if $\{1\}\subset I^+\subset \I\subset \{1,2,3\}\ \wedge 3\in I^+$,}\\
\R_-\times \{0\}&\mbox{if $\{1\}\subset I^+\subset \I\subset \{1,2,3\}\ \wedge 3\in \I\setminus I^+$,}\\
\R\times \{0\}&\mbox{if $\{1\}\subset I^+\subset \I\subset \{1,2\}$.}
\end{cases}\]
Hence we get the inclusions
\begin{eqnarray}
\nonumber\lefteqn{N_{\Gr\widehat N_\Gamma}^2((\yb,\yba);(0,v^\ast))}\\
\label{EqExInclLimRegNormalCone}&\subset&
\begin{cases}(\R\times\R)\times\{(0,0)\}\cup \{(w^\ast,w)\mv w_2=0, w_1^\ast=2w_1 \}&\\
\qquad\cup \{(w^\ast,w)\mv w_1\leq 0, w_2=0, w_1^\ast\geq 2w_1 \}&\mbox{if $v_1^\ast=0$}\\
(\R\times\R)\times\{(0,0)\}&\mbox{if $v_1^\ast>0$}
\end{cases}
\end{eqnarray}
and two-sided estimates
\begin{equation}\label{EqExInclLimNormalCone}L\subset N_{\Gr\widehat N_\Gamma}(\yb,\yba)\subset L\cup (\R\times\R)\times\{(0,0)\}
\end{equation}
with
\begin{eqnarray*}
L&:=&\widehat N_{\Gr\widehat N_\Gamma}(\yb,\yba)\cup N_{\Gr\widehat N_\Gamma}((\yb,\yba);((-1,0),(2,0))\\
&=& \{(w^\ast,w)\mv w_1\leq 0, w_2=0, w_1^\ast\geq 2w_1\}\cup\{(w^\ast,w)\mv w_2=0, w_1^\ast= 2w_1 \}.
\end{eqnarray*}
Let us now compute $N_{\Gr\widehat N_\Gamma}^2((\yb,\yba);(0,v^\ast))$ exactly by the definition. By using Theorem \ref{ThRegNormalCone} we obtain
\[\widehat N_{\Gr \widehat N_\Gamma}(\yb,y^\ast)=\begin{cases}(\R\times\R)\times \{(0,0)\}& \mbox{if $ y_1^\ast>0,y_2^\ast>0$,}\\
\{(w^\ast,w)\mv w_1\leq 0, w_2=0, w_1^\ast\geq 2w_1 \}& \mbox{if $ y_1^\ast=0,y_2^\ast>0$,}\\
\emptyset& \mbox{if $ y_1^\ast<0,y_2^\ast>0$}
\end{cases}\]
and consequently
\begin{eqnarray*}\lefteqn{N_{\Gr\widehat N_\Gamma}^2((\yb,\yba);(0,v^\ast))}\\
&=&
\begin{cases}(\R\times\R)\times\{(0,0)\}\cup \{(w^\ast,w)\mv w_1\leq 0, w_2=0, w_1^\ast\geq 2w_1 \}&\mbox{if $v_1^\ast=0$,}\\
(\R\times\R)\times\{(0,0)\}&\mbox{if $v_1^\ast>0$,}
\end{cases}
\end{eqnarray*}
showing that the inclusion \eqref{EqExInclLimRegNormalCone} is strict in case $v_1^\ast=0$ and that the assertion of Proposition \ref{PropRegContLim} does not hold due to $\Null(\yb)=\{0\}$. Nevertheless, the second inclusion in \eqref{EqExInclLimNormalCone} holds with equality.\hfill$\triangle$
\end{example}

\begin{example}\label{ExampleEq}
Now let $\Gamma\subset\R^2$ be given merely by
$$
q(y)=
\begin{pmatrix}
-y_1^2+y_2\\
-y_1^2-y_2
\end{pmatrix},
$$
$\bar{y}=(0,0)$ and  $\yba=(0,1)$. Again MFCQ is violated at $\yb$, but $M_q$ is metrically subregular at $(0,0)$ and metrically regular in the vicinity of $0$.
Straightforward calculations yield $\Kb=\R\times\{0\}$,
\[
\Lb=\{\lambda \in \mathbb{R}^{2}_{+}\mv \lambda_{1}- \lambda_{2}=1\},
\]
\[
\Lb(v)=\begin{cases}
\{(1,0)\} & \mbox{ if $0 \neq v \in \Kb$,} \\
\bar{\Lambda} & \mbox{ if $v=0$,}
\end{cases}
\]
\[T_{\Gr\widehat{N}_\Gamma}(\yb,\yba )=\{(v,v^\ast)\mv v_2 = 0,\, v_1^\ast=-2v_1\}\]
and
\[\widehat N_{\Gr \widehat N_\Gamma}(\yb,\yba)=\{(w^\ast,w)\mv  \, w_2 = 0,\, w_1^\ast=2w_1\}.\]
Similarly as in Example \ref{ExampleIncl} we obtain for every $0\not=v\in\Kb$, that $\Jb(v)=\{\emptyset,\{1\},\{2\}\}$ and that for every $\J\subset\Jb(v)$ the mapping $(q_i)_{i\in \J}$ is 2-regular in direction $v$.\\
Further, for every $(v,v^\ast)\in T_{\Gr\widehat{N}_\Gamma}(\yb,\yba)$ with $v\not=0$ we have  \begin{eqnarray*}\Mb(v,v^\ast)&=&\{(1,0)\}\times\{(\mu_1,\mu_2)\mv \mu_2\geq0,\ \mu_1-\mu_2=v_2^\ast\}
\end{eqnarray*}
yielding
\[N_{\Gr \widehat N_\Gamma}((\yb,\yba);(v,v^\ast))=Q(v,(1,0),\{1\},\{1\})\]
by Theorem \ref{ThLimNormalCone1}. As in Example \ref{ExampleIncl} we can derive
\[\Kb_{\{1\},\{1\}}(v)=\Kb_{\{1\},\{1\}}=\R\times\{0\},\ Q(v,(1,0),\{1\},\{1\})=\{(w^\ast,w)\mv w_2 = 0, \, w_1^\ast= 2w_1\}\]
and
$N^1_{\Gr \widehat N_\Gamma}((\yb,\yba);(0,v^\ast))=\{(w^\ast,w)\mv w_2 = 0, \, w_1^\ast= 2w_1\}$.\\
Now we consider $N^2_{\Gr \widehat N_\Gamma}((\yb,\yba);(0,v^\ast))$. $\Null(\yb)$ amounts to $\Kb=\R\times\{0\}$, $\bar{\tilde \Lambda}^\E(\tilde v)=\{(1,0)\}$ $\forall \tilde v\in \Null(\yb)$  and
\[\{(\lambda,\mu)\in \Mb(0,v^\ast)\mv \lambda\in\bar{\tilde \Lambda}^\E(\tilde v)\}=\{(1,0)\}\times\{(\mu_1,\mu_2)\mv \mu_2\geq0,\ \mu_1-\mu_2=v_2^\ast\}\ \forall \tilde v\in\Null(\yb),\]
$\Jb(\tilde v)=\{\emptyset,\{1\},\{1\}\}$, $0\not=\tilde v\in\Null(\yb)$ and $\Jb(0)=\{\{1,2\}\}$. Using Proposition \ref{PropRegLimCone} we obtain
\[N_{\Gr\widehat N_\Gamma}^2((\yb,\yba);(0,v^\ast))\subset
Q_0(\tilde v,(1,0),\{1\},\{1\})=\{(w^\ast,w)\mv w_2 = 0, \, w_1^\ast= 2w_1\}.\]
This verifies the inclusion $N_{\Gr\widehat N_\Gamma}^2((\yb,\yba);(0,v^\ast))\subset N_{\Gr\widehat N_\Gamma}^1((\yb,\yba);(0,v^\ast))$ as stated in Proposition \ref{PropRegContLim}. Moreover, all the assumptions of Theorem \ref{ThEqLimNormalCone} are fulfilled and
\[N_{\Gr\widehat N_\Gamma}(\yb,\yba)=\{(w^\ast,w)\mv w_2 = 0, \, w_1^\ast= 2w_1\}.\]
\hfill$\triangle$
\end{example}
Note that the results of Examples \ref{ExampleIncl}, \ref{ExampleEq} cannot be obtained by any technique developed to this purpose so far.

\section{Stability of parameterized equilibria}

In this section we consider a parameter-dependent equilibrium governed by the GE
\begin{equation}\label{eq-5.1}
0 \in F(x,y) + \hat{N}_{\Gamma}(y),
\end{equation}
%%%
where $ x \in \mathbb{R}^{n} $ is the {\em parameter}, $ y \in \mathbb{R}^{m} $ is the  {\em decision variable}, $ F: \mathbb{R}^{n} \times \mathbb{R}^{m} \rightarrow \mathbb{R}^{m}$ is continuously differentiable and $ \Gamma $ is given by \eqref{eq-207}. Our aim is to analyze local stability of the respective  {\em solution map} $ S: \mathbb{R}^{n}\rightrightarrows \mathbb{R}^{m} $ defined by
%%%
\begin{equation}\label{eq-5.2}
S(x):=\{y \in \mathbb{R}^{m} | 0 \in F(x,y)+\hat{N}_{\Gamma}(y)\}
\end{equation}
%%%
around a given {\em reference point} $ (\bar{x},\bar{y})\in {\rm gph} S $. In particular, we will examine the so-called Aubin property of $ S $ around $ (\bar{x},\bar{y}) $ which is an efficient Lipschitz-like property for multifunctions.\\

\begin{definition}[\cite{Aub84}]\label{DefAubProp} $ S $ has the Aubin property around $ (\bar{x},\bar{y}) $ provided there are neighborhoods $ \mathcal{U} $ of $ \bar{x} $, $ \mathcal{V} $ of $ \bar{y} $ and a nonnegative modulus $ \kappa $ such that
%%%
\[
S(x_{1})\cap \mathcal{V} \subset S(x_{2}) + \kappa~\| x_{1}-x_{2} \|~ \B \mbox{ for all } x_{1},x_{2} \in \mathcal{U}.
\]
\end{definition}
%%%
This property can be viewed as a graph localization of the classical local Lipschitz behavior and is closely related to the metric regularity defined in Section 2.

The Aubin property of solution maps has already been investigated in numerous works; let us mention at least \cite[Section 4.4.2]{Mo06a} and \cite{MoOut07}, where the authors have dealt with general
parametric equilibria including (\ref{eq-5.1}) as a special case. In what follows, however, we will confine ourselves with  GE (\ref{eq-5.1}), make use of the results from the preceding section and obtain a new set of conditions ensuring the Aubin property of $S$ around $(\bar{x}, \bar{y})$.

As in the most works about Lipschitz stability our main tool is the Mordukhovich criterion $D^{*}S(\bar{x},\bar{y})(0)=\{0\}$ which is a characterization of the Aubin property around $(\bar{x},\bar{y})$ \cite[Theorem 4.10]{Mo06a}, \cite[Theorem 9.46]{RoWe98}.
In our case it leads directly to the following statement.

\begin{proposition}\label{PropAubProp}
Let the mapping $\hat{N}_{\Gamma}$ have a closed graph around $(\bar{y}, -F(\bar{x},\bar{y}))$ and
assume that the implication
\begin{equation}\label{eq-5.3}
0\in \nabla_{y}F(\bar{x},\bar{y})^{T}b+D^{*}\hat{N}_{\Gamma}(\bar{y},-F(\bar{x},\bar{y}))(b)\Rightarrow b = 0
\end{equation}
holds true.
Then $S$ has the Aubin property around $(\bar{x},\bar{y})$.

If $\nabla_{x}F(\bar{x},\bar{y})$ is surjective, then the above condition  is not only sufficient but also necessary for $S$ to have the Aubin property around $(\bar{x},\bar{y})$.

\end{proposition}

\begin{proof}The first statement is a specialization of \cite[Corollary 4.61]{Mo06a}. The second one follows directly from \cite[Theorem 4.44(i)]{Mo06a}.
\end{proof}

Combining Theorem \ref{ThEqLimNormalCone} with the above statement, we arrive at the following criterion for the Aubin property of $S$ around $(\bar{x},\bar{y})$.

\begin{theorem}\label{ThEquivAub}
Consider GE (\ref{eq-5.1}) and the reference point $(\bar{x},\bar{y})$ and assume that $M_{q}$ is metrically subregular at $(\bar{y},0)$ and  metrically regular in the vicinity of
 $\bar{y}$. Put $\bar{y}^{*}:=-F(\bar{x},\bar{y})$ and suppose that for every nonzero direction $u$ from $\bar{K}(=K(\bar{y},\bar{y}^{*}))$ and every maximal index set $\mathcal{J}\in \bar{\mathcal{J}}(u)$ the mapping
$y \mapsto (q_{i}(y))_{i\in \mathcal{J}}$ is 2-regular at $\bar{y}$ in the direction $u$ and $\mathcal{N}(\bar{y})\neq \{0\}$.

Then the validity of the implication
%%%
\begin{equation}\label{eq-5.6}
-\left[
\begin{array}{c}
\nabla_{y}F(\bar{x},\bar{y})^{T}b\\
b
\end{array}
\right] \in
\bigcup\limits_{\stackrel{(v,v^{*})\in T_{{\rm gph}\hat{N}_{\Gamma}}(\bar{y},\bar{y}^{*})}{v\neq 0}}
\bigcup\limits_{\stackrel{(\lambda, \mu) \in \bar{\mathcal{M}}(v,v^{*})}{\mathcal{J}\in \bar{\mathcal{J}}(v)}}
\bigcup\limits_{I^{+}(\lambda,\mu)\subset I^{+}\subset\mathcal{I}\subset\mathcal{J}} Q(v,\lambda,I^{*},\mathcal{I})\Rightarrow b = 0
\end{equation}
%%%
implies the Aubin property of $S$ around $(\bar{x},\bar{y})$. If $\nabla_{x}F(\bar{x},\bar{y})$ is surjective, then implication (\ref{eq-5.6}) is not only sufficient but also necessary for $S$ to have the Aubin property around $(\bar{x},\bar{y})$.
\end{theorem}
\begin{proof}
The statement follows immediately from Theorem \ref{ThEqLimNormalCone} and  Proposition \ref{PropAubProp}, provided we show that $\mbox{ gph }\hat{N}_{\Gamma}$ is closed around $(\bar{y},\bar{y}^{*})$, i.e., there is a closed ball $B$ around $(\bar{y},\bar{y}^{*})$ such that $\mbox{ gph }\hat{N}_{\Gamma} \cap B$  is closed.
To this aim we will consider sequences
$y_{k}\rightarrow y, y_{k}^{*}\rightarrow y^{*}, y_{k}^{*}\in \hat{N}_{\Gamma}(y_{k})$ with $(y, y^\ast)$  sufficiently close to $(\bar{y},\yba)$.
Note that $M_q$ is metrically subregular at any point $(a,0)$ provided $a\in\Gamma$ is sufficiently close to $\yb$. This implies that
\begin{equation}\label{eq-5.7}
\hat{N}_{\Gamma}(a) = \nabla q(a)^{T} N_{\mathbb{R}^{l}_{-}}(q(a)).
\end{equation}
Let us distinguish among the following three situations:
\begin{enumerate}
\item $y\not=\yb$:  From (\ref{eq-5.7}) we infer the existence of multipliers $\lambda^{k}\in N_{\mathbb{R}^{l}_{-}}(q(y_{k}))$ such that
%%%
\[
y_{k}^{*} = \nabla q (y_{k})^{T}\lambda^{k}.
\]
%%%
By virtue of the assumed  metric regularity of $M_{q}$ in the vicinity of $\bar{y}$ this sequence is bounded, because otherwise the formula for the modulus of metric regularity in \cite[Example 9.44]{RoWe98} would be contradicted. We can thus pass (without relabeling) to a subsequence which converges to  some  $\lambda \in  N_{\mathbb{R}^{l}_{-}}(y)$. It follows that
%%%
\[
y^{*}=\nabla q(y)^{T}\lambda \in \hat{N}_{\Gamma}(y)
\]
and we are done.
\item $y=\yb$ and $y_k=\yb$ at most finitely many times: Then, by passing to a subsequence (without relabeling) one can ensure that $y_k\not=\yb$ $\forall k$ and proceed exactly in the same way as in 1.
\item $y=\yb$ and $y_k=\yb$ infinitely many times:   Then the result follows immediately from  the closedness of $ \hat{N}_{\Gamma}(\bar{y})$.
\end{enumerate}
\end{proof}

We illustrate now the preceding stability criteria by means of two GEs with the constraint sets analyzed in Examples \ref{ExampleIncl} and \ref{ExampleEq}.

\begin{example}\label{old}
Consider the GE (\ref{eq-5.1}) with $x,y \in \mathbb{R}^{2}$ and $F(x,y)=x$.   This GE represents stationarity conditions of the nonlinear program
%%%
\begin{equation}\label{eq-80}
\min_y  \langle y,x\rangle\quad \mbox{ subject to }\quad  y \in \Gamma.
\end{equation}
First let us take $\Gamma$ from Example \ref{ExampleEq} and put $\bar{x}=(0,-1)$, $\bar{y}=(0,0)$.  An application of Proposition   \ref{PropAubProp}  leads to the condition
\[
\{w \in \mathbb{R}^{2}| w_{2}= 0, ~w^{*}_{1} = 2 w_{1}, ~w^{*}_{1}=0\} = \{(0,0)\}
\]
which is clearly fulfilled. Hence, the respective solution map $S$ has the Aubin property around $(\bar{x},\bar{y})$.

Now let us consider the same situation with $\Gamma$ from Example \ref{ExampleIncl}. In this case  the respective solution map would have the Aubin property around $(\bar{x},\bar{y})$ provided   the implication
%%%
\begin{equation}\label{eq-81}
\left[\begin{array}{r}
0\\
w
\end{array}\right] \in L \Rightarrow w = 0,
\end{equation}
holds true. Indeed,  for the second term on the right-hand side of \eqref{EqExInclLimNormalCone} the corresponding implication follows immediately and so it suffices to consider only $L$. Clearly,  (\ref{eq-81}) amounts to
%%%
\[
\left. \aligned w_{1}\leq 0,\\ ~w_{2}=0,\\
~0 \geq 2 w_{1} \endaligned \right\} \Rightarrow w=0.
\]
This implication is, however, clearly violated e.g. by the vector $w=(-1,0)$. Since by virtue of \eqref{EqExInclLimNormalCone} $L$ is a lower estimate of $N_{\Gr\widehat N_\Gamma}(\yb,\yba)$, it follows that the  respective solution map does not possess the Aubin property around $(\bar{x},\bar{y})$. \hfill$\triangle$
\end{example}

\begin{example}\label{new}
Consider again the GE \eqref{eq-5.1} with $x,y \in \mathbb{R}^{2}$ but $F(x,y)=\alpha y-x$, where $\alpha$ is a positive scalar parameter. For $\alpha=1$ this GE represents stationarity conditions of the nonlinear program
%%%
\begin{equation}\label{eq-32}
\mbox{ min } \frac{1}{2} \| y-x \|^{2} ~\mbox{ subject to } y \in \Gamma,
\end{equation}
whose (global) solutions are metric projections of $x$ onto $\Gamma$. As the reference point take $\bar{x}=(0,1), \bar{y}=(0,0)$. With $\Gamma$ from Example \ref{ExampleEq} we obtain the condition
\[
(\alpha w, w)\in \{(w^\ast,w)\mv w_2 = 0, \, w_1^\ast=
2w_1\}\quad\Rightarrow\quad w=0
\]
%%%
which evidently holds true, whenever $\alpha\not=2$. So the Aubin property of the respective $S$ around $(\bar{x},\bar{y})$ has been established for all $\alpha\not=2$.

On the other hand, taking $\Gamma$ from Example \ref{ExampleIncl},
we arrive from \eqref{EqExInclLimNormalCone} at the implication
%%%
\[ (\alpha w,w)
\in L \cup (\mathbb{R}\times \mathbb{R}) \times \{(0,0)\} \Rightarrow w = 0.
\]
%%%
An analysis of this implication tells us that for $\alpha > 2$ the respective solution map does possess the Aubin property around $(\bar{x},\bar{y})$. On the other hand, for $\alpha\leq 2$ there is a nonzero $w$  such that $(\alpha w, w) \in L$. Since  $L$ is a lower estimate of $N_{\Gr\widehat N_\Gamma}(\yb,\yba)$, we conclude that in this case   the solution map does not possess
the Aubin property around $(\bar{x},\bar{y})$. \hfill$\triangle$
\end{example}

 Note that in Example \ref{old} and in Example \ref{new} for $\alpha<2$ $\bar{y}$ is only a stationary point in the optimization problems (\ref{eq-80}), (\ref{eq-32}) for $x = \bar{x}$ but not a minimum. In (\ref{eq-32}) for $x = \bar{x}$ with $\Gamma$ from Example \ref{ExampleIncl} we have to do with 2 stationary points (where the other one ($-0.5 \sqrt{2}, 0.5$) is a (global) minimum). As shown above, the respective $S$ does not behave in a Lipschitzian way around $(\bar{x},\bar{y})$, but on the basis \cite[Theorem 7]{GfrOut14} one can deduce  that it possesses the isolated calmness property at this point.

 \begin{remark}
 In the case of $\Gamma$ from Example \ref{ExampleEq}   in both examples the mappings $S^{-1}$ are even strongly metrically regular at $(\bar{y},\bar{x})$ \cite[page 179]{DR09}. The respective criteria (cf. e.g. \cite{Rob80}, \cite[Section 8.3.4]{KlKum02}), however, cannot be applied, because of a difficult shape of $\Gamma$ around $\bar{y}$.
 \end{remark}

\section{Conclusion}
In the paper we have derived a new technique for the computation of the limiting coderivative of $\widehat N_\Gamma$ for $\Gamma$ given by $C^2$ inequalities. The needed qualification conditions are fairly weak and, in contrast to \cite{LeMo04,MoOut07}, one obtains often exact formulas and not only upper estimates. On the other hand, the computation can be rather demanding, which reflects the complexity of the problem and corresponds to the results obtained for affine inequalities in \cite{HenRoem07}. The results have been used in verifying the Aubin property of parameterized GEs with $\Gamma$ as the constraint set and could be used also in deriving sharp M-stationarity conditions for a class of mathematical programs with equilibrium constraints.

\section*{Acknowledgements}
The research of the first author was supported by the Austrian Science Fund (FWF) under grant P26132-N25. The research of the second author was supported by the Grant Agency of the Czech Republic, project P402/12/1309 and the Australian Research Council, project  DP110102011.
The authors would like to express their gratitude to the reviewer for his/her careful reading and numerous important suggestions.

\end{document}